\documentclass[11pt,a4paper]{amsart}

\usepackage{graphicx}
\usepackage{amssymb}
\usepackage{amsmath}
\usepackage{stmaryrd}
\usepackage{cite}
\usepackage[mathscr]{eucal}
\usepackage{hyperref}
\usepackage[margin=1.1in]{geometry}
\usepackage{comment}

\title[Ricci Solitons from Hopf fibrations]{Cohomogeneity one Ricci Solitons from Hopf fibrations}
\author{Matthias Wink}
\address{Department of Mathematics, UCLA, 520 Portola Plaza, Los Angeles, CA, 90095}

\email{wink@math.ucla.edu}
\thanks{This work was supported by an EPSRC Research Studentship, the Aarhus-Oxford QGM collaboration and the German National Academic Foundation}

\begin{document}

\newcommand{\diam} {\operatorname{diam}}
\newcommand{\Scal} {\operatorname{Scal}}
\newcommand{\scal} {\operatorname{scal}}
\newcommand{\Ric} {\operatorname{Ric}}
\newcommand{\Hess} {\operatorname{Hess}}
\newcommand{\grad} {\operatorname{grad}}
\newcommand{\Rm} {\operatorname{Rm}}
\newcommand{\Rc} {\operatorname{Rc}}
\newcommand{\Curv} {S_{B}^{2}\left( \mathfrak{so}(n) \right) }
\newcommand{ \tr } {\operatorname{tr}}
\newcommand{ \id } {\operatorname{id}}
\newcommand{ \Riczero } {\stackrel{\circ}{\Ric}}
\newcommand{ \ad } {\operatorname{ad}}
\newcommand{ \Ad } {\operatorname{Ad}}
\newcommand{ \dist } {\operatorname{dist}}
\newcommand{ \rank } {\operatorname{rank}}
\newcommand{\Vol}{\operatorname{Vol}}
\newcommand{\dVol}{\operatorname{dVol}}
\newcommand{ \zitieren }[1]{ \hspace{-3mm} \cite{#1}}
\newcommand{ \pr }{\operatorname{pr}}
\newcommand{\diag}{\operatorname{diag}}
\newcommand{\Lagr}{\mathcal{L}}
\newcommand{\av}{\operatorname{av}}

\newtheorem{theorem}{Theorem}[section]
\newtheorem{definition}[theorem]{Definition}
\newtheorem{example}[theorem]{Example}
\newtheorem{remark}[theorem]{Remark}
\newtheorem{lemma}[theorem]{Lemma}
\newtheorem{proposition}[theorem]{Proposition}
\newtheorem{corollary}[theorem]{Corollary}
\newtheorem{assumption}[theorem]{Assumption}
\newtheorem{acknowledgment}[theorem]{Acknowledgment}
\newtheorem{DefAndLemma}[theorem]{Definition and lemma}
\newtheorem{questionroman}[theorem]{Question}

\newenvironment{remarkroman}{\begin{remark} \normalfont }{\end{remark}}
\newenvironment{exampleroman}{\begin{example} \normalfont }{\end{example}}
\newenvironment{question}{\begin{questionroman} \normalfont }{\end{questionroman}}

\renewcommand{\labelenumi}{(\alph{enumi})}
\newtheorem{maintheorem}{Theorem}[]
\renewcommand*{\themaintheorem}{\Alph{maintheorem}}
\newtheorem*{theorem*}{Theorem}
\newtheorem*{corollary*}{Corollary}
\newtheorem*{remark*}{Remark}
\newtheorem*{example*}{Example}
\newtheorem*{question*}{Question}

\newcommand{\R}{\mathbb{R}}
\newcommand{\N}{\mathbb{N}}
\newcommand{\Z}{\mathbb{Z}}
\newcommand{\Q}{\mathbb{Q}}
\newcommand{\C}{\mathbb{C}}
\newcommand{\F}{\mathbb{F}}
\newcommand{\X}{\mathcal{X}}
\newcommand{\D}{\mathcal{D}}
\newcommand{\Cont}{\mathcal{C}}

\begin{abstract}
This paper studies cohomogeneity one Ricci solitons. If the isotropy representation of the principal orbit $G/K$ consists of two inequivalent $\Ad_K$-invariant irreducible summands, the existence of continuous families of non-homothetic complete steady and expanding Ricci solitons on non-trivial bundles is shown. These examples were detected numerically by Buzano-Dancer-Gallaugher-Wang. The analysis of the corresponding Ricci flat trajectories is used to reconstruct Einstein metrics of positive scalar curvature due to B\"ohm. The techniques also apply to $m$-quasi-Einstein metrics. 
\end{abstract}

\maketitle

\section*{Introduction}

A Riemannian manifold $(M,g)$ is called {\em Ricci soliton} if there exists a smooth vector field $X$ on $M$ and a real number  $\varepsilon \in \R$ such that 
\begin{equation*}
\Ric + \frac{1}{2} L_X g + \frac{\varepsilon}{2} g = 0,
\end{equation*}
where $L_X g$ denotes the Lie derivative of the metric $g$ with respect to $X.$ Ricci solitons are generalisations of Einstein manifolds and will be called {\em non-trivial} if $X$ is not a Killing vector field. If $X$ is the gradient of a smooth function $u \colon M \to \R$ then it is called a {\em gradient} Ricci soliton. It is called {\em shrinking}, {\em steady} or {\em expanding} depending on whether $\varepsilon<0,$ $\varepsilon=0$ or $\varepsilon>0.$ Ricci solitons were introduced by Hamilton \cite{HamiltonRFonSurfaces} as self-similar solutions to the Ricci flow and play an important role in its singularity analysis.

\vspace{2mm}

This paper studies the Ricci soliton equation under the assumption of a large symmetry group. For example, Lauret \cite{LauretHomogeneousRS} has constructed non-gradient, homogeneous expanding Ricci solitons. However, Petersen-Wylie \cite{PWRigidityWithSymmetry} have shown that any homogeneous gradient Ricci soliton is rigid, i.e. it is isometric to a quotient of $N \times \R^k,$ where $N$ is Einstein with Einstein constant $\lambda$ and $\R^k$ is equipped with the Euclidean metric and soliton potential $\frac{\lambda}{2} |x|^2.$ 

Therefore it is natural to assume that the Ricci soliton is of {\em cohomogeneity one.} That is, a Lie group acts isometrically on $(M,g)$ and the generic orbit is of codimension one. This will be the setting of this paper. A systematic investigation was initiated by Dancer-Wang \cite{DWCohomOneSolitons} who set up the general framework. Previous examples include the first non-trivial compact Ricci soliton due to Cao \cite{CaoSoliton} and Koiso \cite{KoisoSoliton} or the examples of Feldman-Ilmanen-Knopf \cite{FIKSolitons}, which include the first non-compact shrinking Ricci soliton. It is worth noting that all of these examples, as well as their generalisations due to Dancer-Wang \cite{DWCohomOneSolitons}, are {\em K\"ahler.} In fact, all currently known non-trivial compact Ricci solitons are K\"ahler. On the other hand, Angenent-Knopf \cite{AngenentKnopfRSConicalSingNonuniqueness} constructed non-compact, non-K\"ahler shrinking Ricci solitons.

Hamilton's cigar is also K\"ahler, whereas its higher dimensional analogue, the rotationally symmetric steady soliton on $\R^n$, $n>2,$ the Bryant soliton, is {\em non}-K\"ahler. By extending these examples in a series of papers and then in joint work with Buzano and Gallaugher, Dancer-Wang constructed steady and expanding Ricci solitons of multiple warped product type \cite{DWExpandingSolitons, DWSteadySolitons}, \cite{BDGWExpandingSolitons}, \cite{BDWSteadySolitons}. They also numerically investigated the case where the isotropy representation of the principal orbit $G/K$ consists of two inequivalent $\Ad_K$-invariant irreducible real summands and found numerical evidence for the existence of continuous families of complete steady and expanding Ricci solitons on certain non-trivial vector bundles in \cite{BDGWExpandingSolitons, BDGWSteadySolitons}. This paper gives a rigorous construction thereof:

\vspace{2mm}

Let $G$ be a compact Lie group and let $K \subset H \subset G$ be closed subgroups such that $H/K=S^{d_S}$. Then $H$ acts linearly on $\R^{d_S+1}$ and the associated vector bundle $G \times_H \R^{d_S+1}$ is a cohomogeneity one manifold. Examples where the Lie algebra of $G/K$ decomposes into two inequivalent $\Ad_K$-invariant irreducible real summands include the triples
\begin{align}
(G,H,K) & = (Sp(1) \times Sp(m+1), Sp(1) \times Sp(1) \times Sp(m),  Sp(1) \times Sp(m)), \nonumber \\ 
(G,H,K) & = (Sp(m+1), Sp(1) \times Sp(m), U(1) \times Sp(m)), \label{GroupDiagrams} \\  
(G,H,K) & = (Spin(9), Spin(8), Spin(7)). \nonumber
\end{align}
These examples come from the Hopf fibrations, cf. \cite{BesseEinstein}. In the first and third case, the associated vector bundle is diffeomorphic to $\mathbb{H}P^{m+1} \setminus \left\{ \text{ point } \right\}$ and $CaP^2 \setminus \left\{ \text{ point } \right\},$ respectively. The main theorem is the following:

\begin{maintheorem} 
On $CaP^2 \setminus \left\{ \text{ point } \right\},$ $\mathbb{H}P^{m+1} \setminus \left\{ \text{ point } \right\}$ for $m \geq 1$ and on the vector bundle associated to $(G,H,K) = (Sp(m+1), Sp(1) \times Sp(m), U(1) \times Sp(m))$ for $m \geq 3,$ there exist a $1$-parameter family of non-homothetic complete steady and a $2$-parameter family of non-homothetic complete expanding Ricci solitons.

The steady Ricci solitons are asymptotically paraboloid and thus non-collapsed. The expanding Ricci solitons are asymptotically conical.
\label{MainTheoremTwoSummands}
\end{maintheorem}

Notice that non-trivial gradient steady and expanding Ricci solitons must be non-compact. Furthermore, due to Perelman's \cite{Perelman1} no-local collapsing theorem, blow up limits of finite time Ricci flow singularities are necessarily non-collapsed.

\vspace{2mm}

The construction of the Ricci solitons in Theorem \ref{MainTheoremTwoSummands} partially carries over to the case of complex line bundles over Fano K\"ahler-Einstein manifolds, where Cao \cite{CaoSoliton} and Feldman-Ilmanen-Knopf \cite{FIKSolitons} previously constructed {\em K\"ahler} Ricci solitons. In contrast, Theorem \ref{MainTheoremRSOnLineBundles} exhibits continuous families of complete {\em non}-K\"ahler steady and expanding Ricci solitons.

\begin{maintheorem}
Let $(V,J,g)$ be a Fano K\"ahler-Einstein manifold of real dimension $d$. Suppose that the first Chern class is given by $c_1(V,J) = p \rho$ for an indivisible class $\rho \in H^2(V,J)$ and $\Ric_g = pg.$ For $q \in \Z$ let $\pi \colon P_q \to V$ be the principal circle bundle with Euler class $q \pi^{*} \rho$ and let $L_q$ be the total space of the associated complex line bundle. 

If $2p^2 > (d+2)q^2>0$ there exist a $1$-parameter family of non-homothetic complete steady Ricci solitons and a $2$-parameter family of non-homothetic complete expanding Ricci solitons on $L_q.$ In particular there exist non-K\"ahler Ricci solitons on $L_q.$
\label{MainTheoremRSOnLineBundles}
\end{maintheorem}

In the steady case these Ricci solitons were independently discovered by Stolarski \cite{StolarskiSteadyRSOnCxLineBundles} and Appleton \cite{AppletonSteadyRS}, who use different techniques.

\vspace{2mm}

The proof of Theorem \ref{MainTheoremTwoSummands} establishes that the Ricci soliton metrics correspond to trajectories in a bounded region of a phase space, which implies completeness. This method also applies to Einstein metrics. In particular, in the situation of Theorem \ref{MainTheoremTwoSummands}, the methods of this paper provide an alternative construction of Ricci flat metrics and Einstein metrics with negative scalar curvature due to B\"ohm \cite{BohmNonCompactEinstein}, see also remark \ref{RemarkBoehmSetUpAndProofCompleteness}. 

The associated coordinate change moreover allows good control on the trajectories close to the singular orbit. In the Einstein case this also yields an alternative approach to the following result of B\"ohm \cite{BohmInhomEinstein, BohmNonCompactEinstein}: The two summands Einstein metrics converge to explicit solutions with conical singularities as the volume of the singular orbit tends to zero. In comparison to B\"ohm's work, the main technical simplification is that the methods of this paper do not use the Poincar\'e-Bendixson theorem, see remark \ref{RemarkConvergenceConeSolutions}. As an application, an analysis of the {\em Ricci flat} trajectories will be used to reconstruct Einstein metrics of positive scalar curvature due to B\"ohm \cite{BohmInhomEinstein}. 

\vspace{2mm}

The vector bundles associated to the two families of group diagrams in \eqref{GroupDiagrams} also admit {\em explicit} Ricci flat metrics in the lowest dimensional case $m=1.$ These are in fact of special holonomy $G_2$ and $Spin(7),$ respectively, and were discovered earlier by Bryant-Salamon \cite{BSExceptionalHolonomy} and Gibbons-Page-Pope \cite{GPPEinsteinOnSphereR3R4bundles}. However, it is worth noting that these metrics correspond to {\em linear} trajectories in the above phase space, see theorem \ref{ExplicitRFTrajectories}.

\vspace{2mm}

The techniques in this paper moreover apply if the Bakry-\'Emery Ricci tensor $\Ric + \Hess u$ is replaced with the more general version $\Ric - \Hess u - \frac{1}{m} du \otimes du.$ For any $m \in (0, \infty]$ this leads to the notion of $m$-quasi-Einstein metrics, i.e. Riemannian manifolds which satisfy the curvature condition 
\begin{equation*}
\Ric + \Hess u - \frac{1}{m} du \otimes du + \frac{\varepsilon}{2} g = 0
\end{equation*}
for $u \in C^{\infty}(M)$ and $\varepsilon \in \R.$ These metrics play an important role in the study of Einstein warped products, cf. \cite{CaseSMMSAndQEM} or \cite{HPWUniquenessWarpedProductEinstein} and references therein. 

The initial value problem for cohomogeneity one $m$-quasi-Einstein manifolds will be discussed in the spirit of Eschenburg-Wang \cite{EWInitialValueEinstein} and Buzano \cite{BuzanoInitialValueSolitons}, see theorem \ref{QEMInitialValueTheorem}, and the $m$-quasi-Einstein analogue of Theorem \ref{MainTheoremTwoSummands} is proven in theorem \ref{TwoSummandsQEM}. 

Furthermore, the setting of $m$-quasi Einstein metrics allows a unified proof of the existence of Einstein metrics and Ricci soliton metrics on $\R^{d_1+1} \times M_2 \times \ldots \times M_r,$ for $d_1 \geq 1,$ where $(M_i, g_i)$ are Einstein manifolds with positive scalar curvature. This summaries earlier work due to B\"ohm \cite{BohmNonCompactEinstein}, Dancer-Wang \cite{DWSteadySolitons, DWExpandingSolitons} for $d_1 > 1$ and Buzano-Dancer-Gallaugher-Wang \cite{BDGWExpandingSolitons, BDWSteadySolitons} for $d_1 = 1:$

\begin{maintheorem}
Let $M_2, \ldots, M_r$ be Einstein manifolds with positive scalar curvature and let $d_1 \geq 1$ and $m \in (0, \infty].$ 

Then there is an $(r -1)$-parameter family of non-trivial, non-homothetic, complete, smooth Bakry-\'Emery flat $m$-quasi-Einstein metrics and an $r$-parameter family of non-trivial, non-homothetic, complete, smooth $m$-quasi-Einstein metrics with quasi-Einstein constant $\frac{\varepsilon}{2} > 0$ on $\R^{d_1+1} \times M_2 \times \ldots \times M_r.$
\label{MainTheoremQEM}
\end{maintheorem}

\textit{Structure of the paper.} Section \ref{CohomOneRicciSolitonEQ} reviews the Ricci soliton equation on cohomogeneity one manifolds and recalls some structure theorems. Section \ref{SectionNewSolitons} focuses on the two summands case, with section \ref{SectionSolitonsFromCircleBundles} discussing the case of complex line bundles over Fano K\"ahler-Einstein manifolds. Completeness of the metrics in Theorem \ref{MainTheoremTwoSummands} is shown in section \ref{CompletenessTwoSummands} and the asymptotic behaviour is studied in section \ref{SectionTwoSummandsAsymptotics}. Applications to convergence to cone solutions and B\"ohm's Einstein metrics of positive scalar curvature follow in sections \ref{SectionConvergenceToConeSolutions} and \ref{SectionBohmEinsteinMetricsPosScal}, respectively. Finally, section \ref{SectionQuasiEinsteinMetrics} discusses $m$-quasi-Einstein metrics and the proof of Theorem \ref{MainTheoremQEM}.

\vspace{2mm}

\textit{Acknowledgements.} I wish to thank my PhD advisor Prof. Andrew Dancer for constant support, helpful comments and numerous discussions.

\section{The cohomogeneity one Ricci soliton equation}
\label{CohomOneRicciSolitonEQ}
\subsection{The general set-up}
\label{SectionCohomOneSetUp}

The general framework for cohomogeneity one Ricci solitons has been set up by Dancer-Wang \cite{DWCohomOneSolitons}: Let $(M,g)$ be a Riemannian manifold and let $G$ be a compact connected Lie group which acts isometrically on $(M,g).$ The action is of {\em cohomogeneity one} if the orbit space $M / G$ is one-dimensional. In this case, choose a unit speed geodesic $\gamma \colon I \to M$ that intersects all principal orbits perpendicularly. Let $K=G_{\gamma(t)}$ denote the principal isotropy group. Then $\Phi \colon I \times G/K \to M_0,$ $(t,gK) \mapsto g \cdot \gamma(t)$ is a $G$-equivariant diffeomorphism onto an open dense subset $M_0$ of $M$ and the pullback metric is of the form $\Phi^{*}g=dt^2 + g_t,$ where $g_t$ is a $1$-parameter family of metrics on the principal orbit $P=G/K.$ Let $N = \Phi_*( \frac{\partial}{\partial t})$ be a unit normal vector field and let $L_t = \nabla N$ denote the shape operator of the hypersurface $\Phi(\left\lbrace t\right\rbrace  \times P).$ Via $\Phi,$ $L_t$ can be regarded as a one-parameter family of $G$-equivariant, $g_t$-symmetric endomorphisms of $TP$ which satisfies $\dot{g}_t = 2 g_t L_t.$ Similarly, let $\Ric_t$ be the Ricci curvature corresponding to $g_t.$ According to Eschenburg-Wang \cite{EWInitialValueEinstein} the Ricci curvature of the cohomogeneity one manifold $(M,g)$ is given by 
\begin{align*}
\Ric(X,N) & = - g_t(\delta^{\nabla^t} L_t, X) - d ( \tr(L_t) ) (X), \\
\Ric(N,N) & = -\tr(\dot{L})-\tr(L_t^2), \\
\Ric(X,Y) & = -g_t(\dot{L}(X),Y) - \tr(L_t)g_t(L_t(X),Y) + \Ric_t(X,Y),
\end{align*}
where $X, Y \in TP,$  $ \delta^{\nabla^t} \colon T^*P \otimes TP \to TP$ is the codifferential, and $L_t$ is regarded as a $TP$-valued $1$-form on $TP.$ 
Dancer-Wang \cite{DWCohomOneSolitons} observed that, since $G$ is compact, any cohomogeneity one Ricci soliton induces a Ricci soliton with a $G$-invariant vector field. Hence, in the case of gradient Ricci solitons, the soliton potential can be assumed to be $G$-invariant. The gradient Ricci soliton equation $\Ric + \Hess u + \frac{\varepsilon}{2} g= 0$ then takes the form
\begin{align}
-( \delta^{\nabla^t}L_t)^{\flat} - d(\tr(L_t)) & = 0, \label{CohomOneRSa}\\
- \tr( \dot{L}_t) - \tr(L_t^2) + \ddot{u} + \frac{\varepsilon}{2} & =0, \label{CohomOneRSb}\\
- \dot{L}_t - (- \dot{u} + \tr(L_t)) L_t + r_t + \frac{\varepsilon}{2} \mathbb{I} & = 0, \label{CohomOneRSc}
\end{align}
where $r_t = g_t \circ \Ric_t$ is the Ricci endomorphism, i.e. $g_t(r_t(X),Y)=\Ric_t(X,Y)$ for all $X,Y \in TP.$ Conversely, the above system induces a gradient Ricci soliton on $I \times P$ provided that the metric $g_t$ is defined via $\dot{g}_t = 2 g_t L_t.$ The special case of constant $u$ recovers the cohomogeneity one Einstein equations. 

From now on, for simplicity, the $t$-dependence may not be  stated explicitly.

\vspace{2mm}

It is an immediate consequence of \eqref{CohomOneRSb} that the mean curvature with respect to the volume element $e^{-u}d \Vol_g$ is a Lyapunov function if $\varepsilon \leq 0.$

\begin{proposition}
Fix $\varepsilon \leq 0.$ Then the generalised mean curvature $-\dot{u} + \tr(L)$ 
is monotonically decreasing along the flow of the cohomogeneity one Ricci soliton equation. 
\end{proposition}
If the Ricci soliton metric is at least $C^3$-regular, then the second Bianchi identity implies that the {\em conservation law} 
\begin{equation}
\ddot{u} + (-\dot{u}+\tr(L)) \dot{u} = C+ \varepsilon u
\label{GeneralConservationLaw}
\end{equation}
has to be satisfied for some constant $C \in \R.$ Using the equations \eqref{CohomOneRSb} and \eqref{CohomOneRSc} it can be reformulated as
\begin{equation}
\tr (r) + \tr ( L^2 ) - \left( - \dot{u} + \tr \left( L \right) \right) ^2 + (n-1) \frac{\varepsilon}{2}  = C +\varepsilon u.
\label{ReformulatedGeneralConsLaw}
\end{equation}
Recall that the scalar curvature $R$ of a cohomogeneity one Riemannian manifold $(M^{n+1},g)$ is given by 
$R= \tr(r) - \tr(L^2)- \tr(L)^2 - 2 \tr(\dot{L}).$ Hence it follows with \eqref{CohomOneRSc} that the conservation law \eqref{ReformulatedGeneralConsLaw} is just the cohomogeneity one version of Hamilton's \cite{HamiltonSingularites} general identity $R + | \nabla u |^2 + \varepsilon u  = \overline{C}$ for gradient Ricci solitons (where $\overline{C}= - C -\frac{n+1}{2} \varepsilon$). This also provides a formula for the scalar curvature in terms of the soliton potential:
\begin{equation}
R = - C - \varepsilon u - \dot{u}^2 - (n+1) \frac{\varepsilon}{2}.
\label{ScalarCurvatureRicciSoliton}
\end{equation}

\subsection{Ricci solitons with a singular orbit}
\label{MetricWithASingularOrbit}
From now on, assume that there is a singular orbit $Q = G/H$ at $t=0.$ That is, the orbit at $t=0$ is of dimension strictly less than the dimension of the principal orbit, and let $H=G_{\gamma(0)}$ denote its isotropy group. 

Building up on an idea of Back \cite{BackLocalTheoryofEquiv}, see also \cite{EWInitialValueEinstein}, Dancer-Wang \cite{DWCohomOneSolitons} have shown that in the presence of a singular orbit, equation \eqref{CohomOneRSc} implies \eqref{CohomOneRSa} automatically, provided that the metric is at least $C^2$-regular and the soliton potential is of class $C^3.$ Moreover, if in this case the conservation law \eqref{GeneralConservationLaw} is satisfied, then equation \eqref{CohomOneRSb} holds as well. Conversely, any trajectory of the Ricci soliton equations \eqref{CohomOneRSb}, \eqref{CohomOneRSc} that describes a $C^3$-regular metric with a singular orbit has to satisfy the conservation law \eqref{ReformulatedGeneralConsLaw}. 

The initial value problem for gradient cohomogeneity one Ricci solitons has been considered by Buzano \cite{BuzanoInitialValueSolitons}. Extending Eschenburg-Wang's work \cite{EWInitialValueEinstein} in the Einstein case, under a simplifying, technical assumption, the initial value problem can be solved close to a singular orbit regardless of the soliton being shrinking, steady or expanding. However, the solution may not be unique. For a precise statement, see theorem \ref{QEMInitialValueTheorem}.

Notice that $u(0)=0$ can be assumed, as the Ricci soliton equation is invariant under changing the potential by an additive constant. Furthermore, the existence of a singular orbit at $t=0$ imposes the smoothness condition $\dot{u}(0)=0$ on the soliton potential $u.$ If $d_S$ denotes the dimension of the collapsing sphere at the singular orbit, then the trace of the shape operator grows like $\tr(L) = \frac{d_S}{t} + O(t)$ as $t \to 0.$ Therefore the conservation law \eqref{GeneralConservationLaw} implies $\ddot{u}(0)= \frac{C}{d_S + 1}.$ To summarize:
\begin{equation}
u(0)=0, \ \ \ \dot{u}(0)=0, \ \ \ \ddot{u}(0)=\frac{C}{d_S +1}.
\label{InitialConditionsPotentialFunction}
\end{equation}

The existence of a singular orbit has consequences for the behaviour of the soliton potential. Proposition \ref{PotentialFunctionOfExpandingRSAlongCohomOneFlow} below follows from \cite[Propositions 2.3 and 2.4]{BDWSteadySolitons} and \cite[Proposition 1.11]{BDGWExpandingSolitons}. It should be emphasised that the properties hold along the flow of the Ricci soliton equation and completeness of the metric is not required. 

\begin{proposition}
Along any Ricci soliton trajectory with $\varepsilon \geq 0$ and $C<0$ in \eqref{InitialConditionsPotentialFunction} that corresponds to a cohomogeneity one manifold of dimension $n+1$ with a singular orbit at $t=0,$ for $t > 0$ and as long as the solution exists, the soliton potential satisfies $u(t),$ $\dot{u}(t)<0$ and also $\ddot{u}(t)<0$ if $\varepsilon >0$ or $\varepsilon = 0$ and $L_t \neq 0.$

Furthermore, if $\varepsilon = 0$ and $C \leq 0,$ there holds $\tr(L_t) \leq \frac{n}{t}$ for $t>0$ and as long as the solution exists.
\label{PotentialFunctionOfExpandingRSAlongCohomOneFlow}
\end{proposition}

\begin{remarkroman}
The quantity $\frac{\tr(L)}{-\dot{u}+\tr(L)}$ will appear frequently in later calculations. It is useful to note that it satisfies the differential equation
\begin{equation*}
\frac{d}{dt} \frac{\tr(L)}{-\dot{u}+\tr(L)} = \frac{1}{-\dot{u}+\tr(L)} \left\lbrace  \left( \frac{\tr(L)}{-\dot{u}+\tr(L)} -1 \right)\left( \tr(L^2) - \frac{\varepsilon}{2} \right) + \ddot{u}  \right\rbrace.
\end{equation*}
In particular, in the steady case, proposition \ref{PotentialFunctionOfExpandingRSAlongCohomOneFlow} shows that $\frac{\tr(L)}{-\dot{u}+\tr(L)}$ is monotonically decreasing as long as $-\dot{u}+\tr(L) >0.$ According to proposition \ref{CompleteSteadyRSAsymptotics} below, this is always true if the metric corresponds to a complete steady Ricci soliton. In this case, moreover, it follows that $\frac{\tr(L)}{-\dot{u}+\tr(L)} \to 0$ as $t \to \infty.$
\label{RemarkEvolutionOftrLOverGeneralisedMeanCurvature}
\end{remarkroman}

\subsection{Consequences of completeness}
\label{SectionConsequencesOfCompleteness}
If the solution corresponds to a non-trivial {\em complete} Ricci soliton metric, further restrictions on the asymptotics of the soliton potential and the metric are known. 

In the steady case, according to a result of Chen \cite{ChenStrongUniquenessRF}, the ambient scalar curvature of steady Ricci solitons satisfies $R \geq 0$ with equality if and only if the metric is Ricci flat. Then \eqref{ScalarCurvatureRicciSoliton} implies that $C \leq 0$ is a necessary for completeness and $C=0$ precisely corresponds to the Ricci flat case. Munteanu-Sesum \cite{MSgradientRicciSolitons} have shown that non-trivial complete steady Ricci solitons have at least linear volume growth and Buzano-Dancer-Wang used this to show in \cite[Proposition 2.4 and Corollary 2.6]{BDWSteadySolitons}:
\begin{proposition}
Along any trajectory which corresponds to a non-trivial {\em complete} steady cohomogeneity one Ricci soliton of dimension $n+1$ with a singular orbit at $t=0$ and integrability constant $C<0,$ the estimates 
\begin{align*}
0 < \tr(L) \leq \frac{n}{t}
 \ \text{ and } \ 
0 < - \dot{u} \tr(L) < R < 2 \sqrt{-C} \frac{n}{t} + \frac{n^2}{t^2}
\end{align*}
hold for $t > 0$ and the soliton potential satisfies
\begin{align*}
-\dot{u}(t) \to \sqrt{-C} \ \text{ and } \ \ddot{u}(t) \to 0
\end{align*}
as $t \to \infty.$
\label{CompleteSteadyRSAsymptotics}
\end{proposition}

In the case of expanding Ricci solitons, a similar result of Chen \cite{ChenStrongUniquenessRF} implies that the scalar curvature $R$ of a non-trivial, complete expanding Ricci soliton satisfies $R > - \frac{\varepsilon}{2}(n+1).$ It follows from \eqref{ScalarCurvatureRicciSoliton} that $0 \geq - \dot{u} ^2 > C + \varepsilon u$ holds on any complete expanding Ricci soliton. The smoothness condition \eqref{InitialConditionsPotentialFunction} at the singular orbit therefore requires $C<0$ as a {\em necessary} condition to construct non-trivial, complete expanding Ricci solitons. Conversely, Einstein metrics with negative scalar curvature correspond to trajectories with $C=0.$

Once the Ricci soliton is shown to be complete, it follows from results of Buzano-Dancer-Gallaugher-Wang \cite{BDGWExpandingSolitons} that any non-trivial, complete, gradient expanding Ricci soliton has at least logarithmic volume growth. This has consequences for the asymptotic behaviour of the soliton, see \cite[Equation (1.10) and Proposition 1.18]{BDGWExpandingSolitons}: There exists constants $a_0, a_1 >0$ and a time $t_0>0$ such that for all $t>t_0$
\begin{equation}
| \tr(L_t) |< \sqrt{\frac{n}{2} \varepsilon} \ \text{ and } \ a_1 t + a_0 < - \dot{u}(t) < \frac{\varepsilon}{2} t + \sqrt{-C}
\label{GeneralAsymptoticsExpandingRS}
\end{equation}
i.e. $- \dot{u}$ growths approximately linearly for $t$ large enough. 

\subsection{The B\"ohm functional}
\label{BohmFunctionalSection}

B\"ohm \cite{BohmNonCompactEinstein} introduced the functional $\mathscr{F}_0$ to the study of Einstein manifolds of cohomogeneity one. Subsequently it was considered by Dancer-Wang and their collaborators Buzano, Gallaugher and Hall in the context of cohomogeneity one Ricci solitons \cite{BDWSteadySolitons, BDGWExpandingSolitons, DHWShrinkingSolitons}. The significance of $\mathscr{F}_0$ lies in the fact that it is monotonic under mild assumptions.

To define it, let $v(t) = \sqrt{\det g_t}$ denote the relative volume of the principal orbits and let $L^{(0)} = L - \frac{1}{n}\tr(L) \mathbb{I}$ denote the trace less part of the shape operator. Then the B\"ohm functional is given by
\begin{equation}
\mathscr{F}_0= v ^{\frac{2}{n}} \left( \tr(r_t) + \tr(( L^{(0)})^2 )\right).
\label{BohmFunctional}
\end{equation}

The following proposition is due to Dancer-Hall-Wang \cite[Proposition 2.17]{DHWShrinkingSolitons}. 

\begin{proposition} Along the flow of a $C^3$-regular cohomogeneity one gradient Ricci soliton the B\"ohm functional $\mathscr{F}_0$ satisfies 
\begin{equation}
\frac{d}{dt} \mathscr{F}_0 = - 2 v^{\frac{n}{2}} \tr((L^{(0)})^2) \left( -\dot{u} + \frac{n-1}{n} \tr(L) \right).
\label{DerivativeOfBohmFunctional}
\end{equation}
\label{PropositionBohmFunctional}
\end{proposition}

\begin{remarkroman}
The $C^3$-regularity condition guarantees that the conservation law \eqref{GeneralConservationLaw} is satisfied. On the other hand the existence of a singular orbit along the trajectory is not required to prove \eqref{DerivativeOfBohmFunctional}.
\end{remarkroman}

\section{New Examples of Ricci solitons}
\label{SectionNewSolitons}

\subsection{The geometric set-up}
\label{SectionGeometricSetUp}

Let $(M^{n+1},g)$ be a Riemannian manifold and suppose that $G$ is a compact connected Lie group which acts isometrically on $(M,g)$. Assume that the orbit space is a half open interval and let $K \subset H$ denote the isotropy groups of the principal and singular orbit, respectively. It follows that $M$ is diffeomorphic to the open disc bundle $G \times_H D^{d_S+1} \to G/H,$ where $D^{d_S+1}$ denotes the normal disc to the singular orbit $G/H$ and $S^{d_S} = H/K$ is the collapsing sphere. Conversely, let $G$ be a compact connected Lie group and let $K \subset H$ be closed subgroups such that $H /K$ is a sphere. Then $G \times_H \R^{d_S+1}$ is a cohomogeneity one manifold with principal orbit $G/K.$ Suppose that the non-principal orbit $G/H$ is singular, i.e. of dimension strictly less than $G/K.$

Choose a bi-invariant metric $b$ on $G$ which induces the metric of constant curvature $1$ on $H/K.$ The {\em two summands case} assumes that the space of $G$-invariant metrics on the principal orbit is two dimensional: Let $\mathfrak{g}=\mathfrak{k} \oplus \mathfrak{p}$ be an $\Ad(K)$-invariant decomposition of the Lie algebra of $G$ and suppose furthermore that $\mathfrak{p}$ decomposes into two inequivalent, $b$-orthogonal, irreducible $K$-modules, $\mathfrak{p} = \mathfrak{p}_1 \oplus  \mathfrak{p}_2.$ In fact, $\mathfrak{p}_1$ can be identified with the tangent space to the collapsing sphere $S^{d_S} = H /K$ and $\mathfrak{p}_2$ with the tangent space of the singular orbit $Q=G/H.$ Let $g_S = b_{|\mathfrak{p}_1}$ and $g_Q = b_{|\mathfrak{p}_2}$ denote the induced metrics. Then, away from the singular orbit $Q,$ the metric on $M$ is given by
\begin{equation}
g_{M \setminus Q} = dt^2 +f_1(t)^2 g_S + f_2(t)^2 g_Q
\label{TwoSummandsMetric}
\end{equation} 
and the shape operator of the principal orbit takes the form
\begin{align*}
L_t=\left( \frac{\dot{f}_1}{f_1} \mathbb{I}_{d_1},  \frac{\dot{f}_2}{f_2} \mathbb{I}_{d_2} \right),
\end{align*}
where $d_1=d_S$ is the dimension of the collapsing sphere and $d_2$ is the dimension of the singular orbit. Furthermore, it follows from the theory of Riemannian submersions and the O'Neill calculus, cf. \cite{BohmInhomEinstein}, that the Ricci endomorphism takes the form
\begin{align}
r_t =  \left( \left\lbrace  \frac{A_1}{d_1} \frac{1}{f_1^2} + \frac{A_3}{d_1}   \frac{f_1^2}{f_2^4} \right\rbrace \mathbb{I}_{d_1},
              \left\lbrace  \frac{A_2}{d_2} \frac{1}{f_2^2} - \frac{2 A_3}{d_2} \frac{f_1^2}{f_2^4} \right\rbrace \mathbb{I}_{d_2} \right).
\label{RicciEndomorphismTwoSummands}
\end{align}

Here the constants $A_i \geq 0$ are defined as follows: $A_1 = d_1 (d_1 -1)$, $A_2 = d_2 \Ric^Q,$ where $\Ric^Q$ is the Einstein constant of the isotropy irreducible space $(Q,g_Q),$ and $A_3 = d_2 || A ||^2$, where $ || A || \geq 0$ appears naturally in the theory of Riemannian submersions, cf. \cite{BohmInhomEinstein}: Fix the background metric $g_P=g_S + g_Q$ on the principal orbit $P$ and let $\nabla^{g_P}$ be the corresponding Levi-Civita connection. If $H_1, \ldots, H_{d_2}$ is an orthonormal basis of horizontal vector fields with respect to the Riemannian submersion $(G/K,g_P) \to (G/H,g_Q)$, then $|| A ||^2 = \sum_{i=1}^{d_2} g_S( (\nabla_{H_1}^{g_P} H_i)_{|v}, (\nabla_{H_1}^{g_P} H_i)_{|v} )$ is the norm of an O'Neill tensor associated to the above Riemannian submersion, where $( \cdot )_{|v}$ denotes the projection onto the tangent space of the fibre $S^{d_1}=S^{d_S}.$ 

Warped product metrics with two homogeneous summands provide examples with $||A||=0.$ Examples with $||A|| > 0$ are given by the total spaces of non-trivial disc bundles which are induced by the Hopf fibrations, cf. \cite{BesseEinstein}. The following table, which lists the corresponding group diagrams and associated constants, is taken from \cite[Table 1]{BohmInhomEinstein}.

\vspace{2mm}

\begin{table}[!ht]
$\begin{array}{l|l|l|l|l}
\text{} & \mathbb{C}P^{m+1} & \mathbb{H}P^{m+1} & F^{m+1} & CaP^2  \\ \hline
G & U(m+1) & Sp(1) \times Sp(m+1) & Sp(m+1) & Spin(9) \\
H & U(1) \times U(m) & Sp(1) \times Sp(1) \times Sp(m) & Sp(1) \times Sp(m) & Spin(8) \\
K & U(m) & Sp(1) \times Sp(m) & U(1) \times Sp(m) & Spin(7) \\
d_1 & 1 & 3 & 2 & 7 \\
d_2 & 2m & 4m & 4m & 8 \\
||A||^2 & 1 & 3 & 8 & 7 \\
\Ric^Q & 2m+2 & 4m+8 & 4m+8 & 28
\end{array}$
\caption{Group diagrams associated to Hopf fibrations\label{HopfFibrationsTable}}
\end{table}

\vspace{-4mm}

The soliton potential $u$ will be assumed to be invariant under the action of $G$, $u=u(t),$ and $u(0)=0$ will be fixed. If $u$ satisfies the smoothness conditions \eqref{InitialConditionsPotentialFunction} and the functions $f_1,$ $f_2$ satisfy 
\begin{equation}
f_1(0)=0, \ \dot{f}_1(0)=1 \ \text{ and } \ f_2(0)= \bar{f} > 0, \ \dot{f}_2(0)=0,
\label{SmoothnessMetricTwoSummandsGeometricSetUpSection}
\end{equation}
then the work of Buzano \cite{BuzanoInitialValueSolitons} implies that there is a unique local solution of the Ricci soliton equations with these initial conditions, and it extends the soliton potential and the metric smoothly over the singular orbit.

\begin{remarkroman} 
The two summands case is also the set-up for B\"ohm's work \cite{BohmInhomEinstein, BohmNonCompactEinstein} on Einstein manifolds. In fact, the Lyapunov function \eqref{LyapunovForNonTrivialBundles} is motivated by B\"ohm's work. In contrast, B\"ohm's construction relies on the Poincar\'e-Bendixson theorem. In the Ricci soliton case, however, the extra degree of freedom of the soliton potential does not allow a similar reduction of the Ricci soliton equations to a planar ODE and a new proof is required, see remark \ref{RemarkConvergenceConeSolutions}. Conversely, the methods of section \ref{CompletenessTwoSummands} recover B\"ohm's non-compact Einstein manifolds.
\label{RemarkBoehmSetUpAndProofCompleteness}
\end{remarkroman}

\subsection{Qualitative ODE analysis}
\label{CompletenessTwoSummands}

The Ricci soliton equations for the two summands system can be read off from the discussion in section \ref{SectionGeometricSetUp} and equations \eqref{CohomOneRSb} and \eqref{CohomOneRSc}. However, in this form, the equations become singular at the singular orbit. Therefore, a rescaling will be introduced which smooths the Ricci soliton equation close to the initial value. It was effectively used by Dancer-Wang \cite{DWSteadySolitons} and is motivated by Ivey's work \cite{IveyNewExamplesRS}. Notice that under the coordinate change
\begin{align}
X_i & = \frac{1}{- \dot u + \tr(L)} \frac{\dot{f}_i}{f_i}, \ \ \  \ Y_i  = \frac{1}{- \dot u + \tr(L)} \frac{1}{f_i}, \ \text{for} \ i=1,2, 
\label{RescaledTwoSummandsVariables} 
\\
\mathcal{L} & = \frac{1}{- \dot u + \tr(L)}, \ \ \ \ \ \ \frac{d}{ds} = \frac{1}{-\dot{u}+\tr(L)} \frac{d}{dt} \nonumber
\end{align}
the cohomogeneity one two summands Ricci soliton equations reduce to the ODE system
\begin{align}
X_1^{'} & = X_1 \left(  \sum_{i=1}^2 d_i X_i^2 - \frac{\varepsilon}{2} \mathcal{L}^2 -1 \right)  + \frac{A_1}{d_1}  Y_1^2+\frac{\varepsilon}{2} \mathcal{L}^2 + \frac{A_3}{d_1} \frac{Y_2^4}{Y_1^2}, 
\label{RescaledTwoSummandsODE}
\\
X_2^{'} & = X_2 \left(  \sum_{i=1}^2 d_i X_i^2 - \frac{\varepsilon}{2} \mathcal{L}^2 -1 \right)  + \frac{A_2}{d_2}  Y_2^2+\frac{\varepsilon}{2} \mathcal{L}^2 -  \frac{2 A_3}{d_2} \frac{Y_2^4}{Y_1^2}, \nonumber \\
Y_j^{'} & =Y_j  \left(  \sum_{i=1}^2 d_i X_i^2 - \frac{\varepsilon}{2} \mathcal{L}^2 - X_j \right), \nonumber \\
\mathcal{L}^{'} & =\mathcal{L} \left(  \sum_{i=1}^2 d_i X_i^2 - \frac{\varepsilon}{2} \mathcal{L}^2 \right). \nonumber
\end{align}
Here and in the following, the $\frac{d}{ds}$ derivative is denoted by a prime $'.$ On the other hand, the $\frac{d}{dt}$ derivative will always correspond to a dot $\dot{}$ .

To establish some basic properties of this ODE system, it will be enough to assume that $d_1, d_2 > 0,$ $A_1, A_2 > 0$ and $A_3 \geq 0.$ However, in the main body of the paper $d_1 > 1$ and $A_1, A_2, A_3 > 0$ will be assumed. 

\begin{remarkroman}
(a) The case $A_3 = 0$ is already well understood from works on multiple warped products, see \cite{IveyNewExamplesRS}, \cite{GKExpandingRS}, \cite{DWExpandingSolitons, DWSteadySolitons}, \cite{BDGWExpandingSolitons}, \cite{BDWSteadySolitons} and \cite{AngenentKnopfRSConicalSingNonuniqueness}. 

(b) The case $d_1=1$ implies $A_1=0$ in geometric applications. In this case Cao-Koiso \cite{CaoSoliton},\cite{KoisoSoliton} and Feldman-Ilmanen-Knopf \cite{FIKSolitons} found explicit solutions to the associated {\em K\"ahler} Ricci solitons equations. {\em Non-K\"ahler} steady and expanding Ricci solitons will be constructed in section \ref{SectionSolitonsFromCircleBundles}. In the steady case these were independently found by Stolarski \cite{StolarskiSteadyRSOnCxLineBundles} and Appleton \cite{AppletonSteadyRS}, who use different techniques.
\end{remarkroman}

Notice that time, metric and soliton potential can be recovered from the ODE via
\begin{align*}
t(s) = t(s_0) + \int_{s_0}^{s} \mathcal{L}( \tau ) d \tau \ \text{ and } \ f_i  = \frac{\mathcal{L}}{Y_i}, \ \text{for } \ i=1,2, \ \text{ and } \ \dot{u} = \frac{\sum_{i=1}^2 d_i X_i - 1}{\mathcal{L}}. 
\end{align*}

In the new coordinate system, the smoothness conditions for the metric in \eqref{SmoothnessMetricTwoSummandsGeometricSetUpSection} and the soliton potential in \eqref{InitialConditionsPotentialFunction} correspond to the stationary point
\begin{align}
X_1 = Y_1 = \frac{1}{d_1} \ \text{ and } \ X_2  = Y_2 = 0 \ \text{ and } \ \mathcal{L} =0.
\label{InitialCriticalPoint}
\end{align}
Trajectories emanating from \eqref{InitialCriticalPoint} will be parametrised so that \eqref{InitialCriticalPoint} corresponds to $s = - \infty.$ 

The conservation law \eqref{ReformulatedGeneralConsLaw} takes the form
\begin{equation}
\sum_{i=1}^2 d_i X_i^2 + \sum_{i=1}^2 A_i Y_i^2 - A_3 \frac{Y_2^4}{Y_1^2} + (n-1)\frac{\varepsilon}{2} \mathcal{L}^2 = 1 + \left( C + \varepsilon u \right) \mathcal{L}^2.
\label{GeneralTwoSummandsConsLaw}
\end{equation}

Consider the functions 
\begin{align*}
\mathcal{S}_1 & = \sum_{i=1}^2 d_i X_i^2 + \sum_{i=1}^2 A_i Y_i^2 - A_3 \frac{Y_2^4}{Y_1^2} + (n-1)\frac{\varepsilon}{2} \mathcal{L}^2 -1, \\
\mathcal{S}_2 & = \sum_{i=1}^2 d_i X_i -1.
\end{align*}

Notice that $\mathcal{S}_1$ occurs in the conservation law and $\mathcal{S}_2 = \frac{\dot{u}}{- \dot{u} + \tr(L)}$ encodes the derivative of the soliton potential in the rescaled coordinates.

Fix $\varepsilon \geq 0$ and recall from section \ref{SectionConsequencesOfCompleteness} that $C \leq 0$ is a necessary condition to obtain trajectories that correspond to complete steady or expanding Ricci solitons and that $C=0$ is the Einstein case. Due to the initial conditions \eqref{InitialConditionsPotentialFunction} and proposition \ref{PotentialFunctionOfExpandingRSAlongCohomOneFlow}, the soliton potential satisfies $u, \dot{u} \leq 0$ if $C \leq 0,$ and away from the singular orbit equality can only occur in the Einstein case. 

Therefore, any trajectory with $\varepsilon \geq 0$ and $C \leq 0$ satisfies $\mathcal{S}_1, \mathcal{S}_2 \leq 0.$ Equality occurs at the initial stationary point \eqref{InitialCriticalPoint} and then {\em Einstein} trajectories lie in the locus 
\begin{equation}
\left\{ \mathcal{S}_1 = 0\right\} \cap \left\{ \mathcal{S}_2 = 0\right\}
\label{EinsteinLocus}
\end{equation}
whereas trajectories of {\em complete non-trivial} Ricci solitons are contained in the locus
\begin{equation}
\left\{ \mathcal{S}_1 < 0\right\} \cap \left\{ \mathcal{S}_2 < 0\right\}.
\label{SolitonLocus}
\end{equation}

Conversely, trajectories in these loci correspond to Einstein metrics and non-trivial Ricci solitons. 

The invariance of the above loci for $\varepsilon \geq 0$ follows from the Ricci soliton ODE as a direct calculation verifies
\begin{align*}
\frac{1}{2} \frac{d}{ds} \mathcal{S}_1 & = \left( \sum_{i=1}^2 d_i X_i^2 - \frac{\varepsilon}{2} \mathcal{L}^2 \right) \mathcal{S}_1 + \frac{\varepsilon}{2} \mathcal{L}^2 \cdot \mathcal{S}_2 , \\
 \frac{d}{ds} \mathcal{S}_2 & = \mathcal{S}_1 + \left(  \sum_{i=1}^2 d_i X_i^2 - \frac{\varepsilon}{2} \mathcal{L}^2 - 1 \right) \mathcal{S}_2.
\end{align*}

Now the existence of trajectories which lie in one of the above loci and in the unstable manifold of the critical point \eqref{InitialCriticalPoint} will be discussed. Different trajectories will correspond to non-homothetic Einstein or Ricci soliton metrics.

The linearisation of the Ricci soliton ODE at the initial stationary point \eqref{InitialCriticalPoint} is given by
\begin{equation*}
\begin{pmatrix}
\frac{3}{d_1}-1 & 0 & \frac{2(d_1-1)}{d_1} & 0 & 0 \\
0 & \frac{1}{d_1}-1 & 0 & 0 & 0 \\
\frac{1}{d_1} & 0 & 0 & 0 & 0 \\
0 & 0 & 0 & \frac{1}{d_1} & 0 \\
0 & 0 & 0 & 0 & \frac{1}{d_1}
\end{pmatrix}.
\end{equation*}

The corresponding eigenvalues are hence $\frac{2}{d_1},$ and both $\frac{1}{d_1}-1$ and $\frac{1}{d_1}$ appear twice. In particular, the critical point is {\em hyperbolic} if $d_1 >1.$ The corresponding eigenspaces are given by $E_{\frac{2}{d_1}} = \operatorname{span} \left\lbrace (2,0,1,0,0) \right\rbrace ,$ $E_{\frac{1}{d_1}-1}= \operatorname{span} \left\lbrace (0,1,0,0,0), (d_1-1,0,-1,0,0) \right\rbrace$ and $E_{\frac{1}{d_1}} = \operatorname{span} \left\lbrace (0,0,0,1,0), (0,0,0,0,1) \right\rbrace.$ Notice that the stationary point \eqref{InitialCriticalPoint} lies in the set $\left\{ \mathcal{S}_1 = 0\right\} \cap \left\{ \mathcal{S}_2 = 0\right\}.$ Furthermore, $\left\{ \mathcal{S}_1 = 0\right\}$ is a submanifold of $\R^5$ if $Y_1 \neq 0$ and its tangent space at \eqref{InitialCriticalPoint} is $\operatorname{span} \left\lbrace (1,0,d_1-1,0,0) \right\rbrace^{\perp}.$ Similarly, $\left\{ \mathcal{S}_2 = 0\right\}$ is a submanifold with tangent space $\operatorname{span} \left\lbrace (d_1,d_2,0,0,0) \right\rbrace^{\perp}$ at \eqref{InitialCriticalPoint}. Notice that both tangent spaces contain $E_{\frac{1}{d_1}}$ but not $E_{\frac{2}{d_1}}$ and that $E_{\frac{1}{d_1}} \oplus E_{\frac{2}{d_1}}$ is the tangent space to the unstable manifold.

According to the above discussion, trajectories in the unstable manifold of \eqref{InitialCriticalPoint} that either remain in the set $\left\{ \mathcal{S}_1 = 0\right\} \cap \left\{ \mathcal{S}_2 = 0\right\}$ or flow into $\left\{ \mathcal{S}_1 < 0\right\} \cap \left\{ \mathcal{S}_2 < 0\right\}$ need to be considered. Notice, however, that if $\varepsilon = 0$ the ODE for $\mathcal{L}$ decouples. Hence, the soliton system effectively reduces to a system in $X_i, Y_i$ for $i=1,2.$ Counting trajectories with respect to the possibly reduced system then gives the following result. 

\begin{proposition}
Suppose that $d_1 > 1.$ If $\varepsilon \neq 0,$ then there exists a $1$-parameter family of trajectories lying both in the unstable manifold of \eqref{InitialCriticalPoint} and the Einstein locus \eqref{EinsteinLocus} and a $2$-parameter family of trajectories lying both in the unstable manifold of \eqref{InitialCriticalPoint} and the Ricci soliton locus \eqref{SolitonLocus}.

If $\varepsilon =0,$ then the unstable manifold of \eqref{InitialCriticalPoint} with respect to the reduced two summands ODE in $X_1, X_2$ and $Y_1, Y_2$ contains a unique trajectory lying in the Einstein locus \eqref{EinsteinLocus} and a $1$-parameter family of trajectories lying in the Ricci soliton locus \eqref{SolitonLocus}. These give rise to an (up to scaling) unique Ricci flat metric and a $1$-parameter family of Ricci solitons with soliton potential $u=0$ at the singular orbit.
\label{NumberOfParameterFamilies}
\end{proposition}

Proposition \ref{NumberOfParameterFamilies} is in agreement with the theory of solutions to the initial value problem for cohomogeneity one Ricci solitons and Einstein metrics developed by Buzano \cite{BuzanoInitialValueSolitons} and Eschenburg-Wang \cite{EWInitialValueEinstein}, respectively. Their methods also carry over to the case $d_1 = 1.$ 

\vspace{2mm}

Notice that the ODE system \eqref{RescaledTwoSummandsODE} and the initial stationary point \eqref{InitialCriticalPoint} are invariant under changing the signs of $Y_2,$ $\mathcal{L}.$ Since $\mathcal{L}^{-1}=-\dot{u} + \tr(L) \to + \infty$ as $t \to 0,$ $\mathcal{L}>0$ will be assumed along the trajectories. The choice $f_2(0)= \bar{f}>0$ in \eqref{SmoothnessMetricTwoSummandsGeometricSetUpSection} implies $f_2(t)>0$ for small $t>0$ and thus $Y_2 > 0$ will be assumed. Recall that $\lim_{s \to - \infty} Y_1(s)=1.$ The ODEs for $Y_1, Y_2, \mathcal{L}$ imply that positivity of the variables is preserved along the flow.

The following lemma shows a basic dynamical property of the Ricci soliton ODE and sets up the discussion of the long time behaviour.

\begin{lemma}
Let $\varepsilon \geq 0$ and consider a trajectory of the two summands Ricci soliton ODE that emanates from \eqref{InitialCriticalPoint} at $s = - \infty$ and enters either \eqref{EinsteinLocus} or \eqref{SolitonLocus}.

Then there holds $X_1 >0$ for all finite $s$ and $X_2$ is positive for sufficiently negative $s.$ Moreover, suppose there is an $s_0 \in \R$ such that $X_2(s_0) <0.$ Then $X_2(s) < 0$ for all $s \geq s_0.$ 
\label{XVariablesPositiveInitially}
\end{lemma}
\begin{proof}
Recall that $\lim_{s \to - \infty} X_1 = 1/d_1 >0$ and in particular $X_1$ is positive initially. If there is an $s \in \R$ such that $X_1(s) = 0,$ then $X_1^{'}(s)>0.$ By continuity this implies $X_1 > 0$ everywhere.

The conservation law \eqref{GeneralTwoSummandsConsLaw} implies that $\sum_{i=1}^2 d_i X_i^2 -1 \leq A_3 \frac{Y_2^4}{Y_1^2} - \sum_{i=1}^2 A_i Y_i^2 < 0$ close to \eqref{InitialCriticalPoint} as $Y_1 \to \frac{1}{d_1}$ and $Y_2 \to 0.$ Similarly, $\frac{A_2}{d_2}  Y_2^2 -  \frac{2 A_3}{d_2} \frac{Y_2^4}{Y_1^2} >0$ for sufficiently negative times. If $X_2(s_0) < 0$ in this region, then the ODE
\begin{align*}
X_2^{'} = X_2 \left(  \sum_{i=1}^2 d_i X_i^2 - \frac{\varepsilon}{2} \mathcal{L}^2 -1 \right)  + \frac{A_2}{d_2}  Y_2^2+\frac{\varepsilon}{2} \mathcal{L}^2 -  \frac{2 A_3}{d_2} \frac{Y_2^4}{Y_1^2}
\end{align*}
implies that $X_2^{'}(s_0) > 0$ as $\varepsilon \geq 0.$ In particular $X_2(s) \leq X_2(s_0) < 0$ for all $s \leq s_0.$ This contradicts $X_2 \to 0$ as $s \to - \infty.$ 

If the last statement is not true, then there exist $s_{*} < s^{*}$ such that $X_2 < 0$ on $(s_{*}, s^{*})$ and
\begin{align*}
X_2(s_{*}) & = 0 \ \text{ and } \ X_2^{'}(s_{*}) \leq 0, \\
X_2(s^{*}) & = 0 \ \text{ and } \ X_2^{'}(s^{*}) \geq 0.
\end{align*}
It follows that $ \frac{A_2}{d_2} Y_2^2(s_{*}) + \frac{\varepsilon}{2} \mathcal{L}^2(s_{*}) - 2 \frac{A_3}{d_2} \frac{Y_2^4}{Y_1^2}(s_{*}) \leq 0$ which is equivalent to 
\begin{equation*}
\frac{A_2}{d_2} \leq \left[ 2 \frac{A_3}{d_2}  \left( \frac{Y_2}{Y_1} \right) ^2 - \frac{\varepsilon}{2} \left( \frac{\mathcal{L}}{Y_2} \right) ^2 \right](s_{*}).
\end{equation*}
Similarly, the second condition implies the reverse inequality at $s^{*}.$ Therefore,
\begin{align*}
0 & \leq \left[ 2 \frac{A_3}{d_2}  \left( \frac{Y_2}{Y_1} \right) ^2 - \frac{\varepsilon}{2} \left( \frac{\mathcal{L}}{Y_2} \right) ^2 \right](s_{*}) - \left[ 2 \frac{A_3}{d_2}  \left( \frac{Y_2}{Y_1} \right) ^2 - \frac{\varepsilon}{2} \left( \frac{\mathcal{L}}{Y_2} \right) ^2 \right](s^{*}) \\
 & = \frac{d}{ds} \left[ 2 \frac{A_3}{d_2}  \left( \frac{Y_2}{Y_1} \right) ^2 - \frac{\varepsilon}{2} \left( \frac{\mathcal{L}}{Y_2} \right) ^2 \right](\xi) \cdot (s_{*} - s^{*} )
\end{align*}
for some $\xi \in (s_{*}, s^{*}).$ On the other hand, observe that
\begin{align*}
\frac{d}{ds} \frac{Y_2}{Y_1} = \frac{Y_2}{Y_1} (X_1-X_2) \ \text{ and } \
\frac{d}{ds} \frac{\mathcal{L}}{Y_2} = \frac{\mathcal{L}}{Y_2} X_2.
\end{align*}
Therefore, $X_2( \xi) < 0,$ $\varepsilon \geq 0$ and $s_{*} < s^{*}$ imply
\begin{align*}
0 & \leq \frac{d}{ds} \left[ 2 \frac{A_3}{d_2}  \left( \frac{Y_2}{Y_1} \right) ^2 - \frac{\varepsilon}{2} \left( \frac{\mathcal{L}}{Y_2} \right) ^2 \right](\xi) \cdot (s_{*} - s^{*} ) \\
& = \ 2 \ \left[ 2 \frac{A_3}{d_2}  \left( \frac{Y_2}{Y_1} \right) ^2 (X_1-X_2) - \frac{\varepsilon}{2} \left( \frac{\mathcal{L}}{Y_2} \right) ^2 X_2 \right](\xi) \cdot (s_{*} - s^{*} ) < 0,
\end{align*}
which is a contradiction.
\end{proof}

\begin{remarkroman}
In fact, the possibility that $X_2 < 0$ is the only obstruction to long time existence. Geometrically this says that along the trajectory of an incomplete metric the shape operator cannot remain positive definite. 

If $A_3=0,$ then $X_2 > 0$ is immediate and the Einstein and Ricci soliton loci \eqref{EinsteinLocus} and \eqref{SolitonLocus}, respectively, are bounded regions in phase space. Completeness of the metric then follows as in propositions \ref{CompletenessEpsZeroTwoSummands} and \ref{CompletenessEpsPosTwoSummands} below. Geometrically the case $A_3=0$ corresponds to the doubly warped product situation which was considered by Ivey \cite{IveyNewExamplesRS}, Gastel-Kronz  \cite{GKExpandingRS}, Dancer-Wang \cite{DWExpandingSolitons, DWSteadySolitons} and Angenent-Knopf \cite{AngenentKnopfRSConicalSingNonuniqueness}.
\end{remarkroman}

If $A_3>0,$ notice that $X_2 > 0$ clearly holds as long as $\frac{Y_2}{Y_1} < \sqrt{\frac{A_2}{2 A_3}}.$ Therefore, the quotient
\begin{equation*}
\omega = \frac{Y_2}{Y_1}.
\end{equation*}
plays a central role in the discussion. Observe that $\omega$ satisfies 
\begin{align}
\omega{'} = \omega (X_1-X_2).
\label{DByDsOmega}
\end{align}
In fact this implies that the Ricci soliton equation is equivalent to an ODE system with polynomial right hand side.

In order to obtain an a priori bound for $\omega,$ fix $d_1>1$ and consider the function
\begin{align}
\widehat{\mathcal{G}}(\omega) = \frac{A_1}{d_1} \frac{\omega^{2(d_1-1)}}{2(d_1-1)}  -  \frac{A_2}{d_2} \frac{\omega^{2d_1}}{2d_1} + A_3 \left( \frac{1}{d_1} + \frac{2}{d_2} \right) \frac{\omega^{2(d_1+1)}}{2(d_1+1)}.
\label{FunctionGHatTwoSummandsCase}
\end{align}
Along trajectories of the two summands Ricci soliton ODE there holds
\begin{align*}
\frac{d}{ds} \widehat{\mathcal{G}}(\omega) = \omega^{2(d_1-1)} \left\{ \frac{A_1}{d_1} - \frac{A_2}{d_2} \omega^2 + A_3 \left( \frac{1}{d_1} + \frac{2}{d_2} \right) \omega^4  \right\} \left( X_1 - X_2 \right)
\end{align*}
and non-zero roots of $\widehat{\mathcal{G}}$ are of the form
\begin{align*}
\omega^2 = \frac{1}{2} \frac{A_2}{A_3} \frac{d_1+1}{2d_1+d_2} \left\lbrace 1 \pm \sqrt{1 - 4 \frac{A_1 A_3}{A_2^2} \frac{d_2 (2d_1+d_2)}{(d_1-1)(d_1 + 1)}} \right\rbrace.
\end{align*}
In particular, there exist two positive roots $0 < \hat{\omega}_1 < \hat{\omega}_2$ if and only if 
\begin{equation}
\widehat{D} = \frac{A_2^2}{d_2^2} - 4 \frac{A_1}{d_1(d_1-1)} \frac{A_3}{d_2} \frac{d_1}{d_1+1} (2d_1 +d_2) > 0.
\label{DefinitionDHat}
\end{equation}
Moreover, in this case, $\hat{\omega}_1^2 < \frac{A_2}{2 A_3}.$
\begin{proposition} Suppose that $d_1 > 1,$ $\widehat{D} >0$ and $\varepsilon \geq 0.$ Then the set
\begin{align*}
\left\{ \ X_2 > 0 \ \text{ and } \ 0 < \frac{Y_2}{Y_1} < \hat{\omega}_1 \ \right\}
\end{align*}
contains any trajectory of the two summands Ricci soliton ODE that emanates from \eqref{InitialCriticalPoint} and flows into either \eqref{EinsteinLocus} or \eqref{SolitonLocus}.
\label{X2VariablePositive}
\end{proposition}
\begin{proof}
The ODE for $X_2$ shows that $X_2$ remains positive if $\frac{Y_2^2}{Y_1^2}  = \omega^2 < \frac{A_2}{2 A_3}.$ Since $\hat{\omega}_1^2 < \frac{A_2}{2 A_3},$ it suffices to show that $\omega < \hat{\omega}_1$ as long as $X_2 > 0.$ Consider the function
\begin{align} 
\mathcal{K} = \frac{1}{2} \omega^{2(d_1-1)} \left( \frac{X_1 - X_2}{Y_1} \right)^2 - \widehat{\mathcal{G}} \left( \omega \right),
\label{LyapunovForNonTrivialBundles}
\end{align}
which was introduced by B\"ohm in the Einstein case \cite{BohmInhomEinstein}.
On the set $X_2 >0$ it is a Lyapunov function since
\begin{align*}
\frac{d}{ds} \mathcal{K} = \omega^{2(d_1-1)} \left( \frac{X_1 - X_2}{Y_1} \right)^2 \left\{ \sum_{i=1}^2 d_i X_i - 1 - (n-1) X_2 \right\}
\end{align*}
and $\sum_{i=1}^2 d_i X_i - 1 \leq 0$ holds in both loci. Notice that $\lim_{s \to - \infty} \mathcal{K} = 0$ and $\mathcal{K} \geq 0$ if $\frac{Y_2}{Y_1} = \omega = \hat{\omega}_1.$ However, $\mathcal{K}$ is non-increasing and strictly decreasing close to \eqref{InitialCriticalPoint}. This completes the proof.
\end{proof}

\begin{corollary}
Suppose that $d_1 >1,$ $\widehat{D} > 0$ and $\varepsilon \geq 0.$  Then along trajectories emanating from \eqref{InitialCriticalPoint} and flowing into \eqref{EinsteinLocus} or \eqref{SolitonLocus} there holds $X_1,$ $X_2 >0$ for all finite times. Moreover, the variables $X_1, X_2$ and $Y_1, Y_2$ and $\omega$ are bounded, and if $\varepsilon >0$ then $\mathcal{L}$ is bounded too. In particular, the rescaled flow exists for all times.
\label{FlowExistsForAllTimes}
\end{corollary}
\begin{proof}
According to lemma \ref{XVariablesPositiveInitially} one has $X_1,$ $X_2 > 0$ initially and $X_1 > 0$ is preserved along the flow. Positivity of $X_2$ follows from proposition \ref{X2VariablePositive} and $X_1,$ $X_2$ remain bounded as $0 \leq d_1 X_1 + d_2 X_2 \leq 1$ due to \eqref{EinsteinLocus} and \eqref{SolitonLocus}.

Then the ODE for $\mathcal{L}$ implies that $\mathcal{L}$ cannot blow up in finite time as $\varepsilon \geq 0.$ By the same argument, this also holds for $Y_1,$ $Y_2.$

Alternatively, it follows from the bound $\frac{Y_2^2}{Y_1^2} < \hat{\omega}_1^2 < \frac{A_2}{2 A_3}$ that $A_2 - A_3 \frac{Y_2^2}{Y_1^2} > \frac{A_2}{2}.$ The Einstein and Ricci soliton loci \eqref{EinsteinLocus} and \eqref{SolitonLocus} are therefore contained in the bounded region $\left\lbrace \ \sum_{i=1}^2 d_i X_i^2 +  A_1 Y_1^2 + \frac{A_2}{2} Y_2^2 + (n-1) \frac{\varepsilon}{2} \mathcal{L}^2 \leq 1 \ \right\rbrace.$ By considering $\omega = \frac{Y_2}{Y_1}$ as an independent variable, one obtains an ODE system with polynomial right hand side. Since $\omega < \hat{\omega}_1$ is bounded, standard ODE theory implies that the flow exists for all times.
\end{proof}

In order to prove that the corresponding metrics are complete, it suffices to show that $t_{\max} = \infty.$ Recall from the coordinate change that
\begin{equation}
t(s) = t(s_0) + \int_{s_0}^s \mathcal{L} (\tau) d \tau.
\label{TimeRescaling}
\end{equation}
Therefore it is necessary to estimate the asymptotic behaviour of $\mathcal{L}.$ This needs to be considered separately for the cases $\varepsilon = 0$ and $\varepsilon >0.$

\begin{proposition}
Suppose that $d_1>1$ and $\widehat{D}>0.$ Then the corresponding steady Ricci soliton and Ricci flat metrics are complete.
\label{CompletenessEpsZeroTwoSummands}
\end{proposition}
\begin{proof}
A special feature of the case $\varepsilon = 0$ is that $\mathcal{L} = \frac{1}{-\dot{u}+\tr(L)}$ is in fact a {\em Lyapunov} function.  As $\mathcal{L}$ becomes positive initially and is therefore monotonically increasing, it is bounded away from zero for $s \geq s_0$ and any $s_0 \in \R.$ Then the time rescaling \eqref{TimeRescaling} shows that $t \to \infty$ as $s \to \infty,$ i.e. the metrics are complete.
\end{proof}

\begin{remarkroman}
The Ricci flat metrics in proposition \ref{CompletenessEpsZeroTwoSummands} have already been constructed by B\"ohm \cite{BohmNonCompactEinstein} by different means, see remark \ref{RemarkConvergenceConeSolutions}. 
\end{remarkroman}

The cases of expanding Ricci solitons and Einstein metrics with negative scalar curvature correspond to  $\varepsilon >0.$ It will be sufficient to have an upper bound on $\mathcal{L}$ to prove completeness of the metric. 

\begin{lemma} Let $d_1>1,$ $\widehat{D} > 0$ and $\varepsilon > 0.$ Then along trajectories that emanate from \eqref{InitialCriticalPoint} and flow into \eqref{EinsteinLocus} or \eqref{SolitonLocus} there holds 
\begin{equation*}
0 < \frac{\varepsilon}{2} \mathcal{L}^2 \leq \max \left\lbrace \frac{1}{d_1}, \frac{1}{d_2} \right\rbrace.
\end{equation*}
Moreover, in the Einstein case, given $s_0 \in \R$ there holds
\begin{equation*}
\frac{\varepsilon}{2} \mathcal{L}^2(s) \geq \min \left\lbrace \frac{\varepsilon}{2} \mathcal{L}^2(s_0), \frac{1}{n} \right\rbrace 
\end{equation*}
for all $s \geq s_0.$ In particular, $\mathcal{L}(s)$ is bounded away from zero for all $s \geq s_0.$
\label{LemmaLBoundedAwayFromZero}
\end{lemma}
\begin{proof}
Notice that $0 \leq \sum_{i=1}^2 d_i X_i^2 \leq \max \left\lbrace \frac{1}{d_1}, \frac{1}{d_2} \right\rbrace $ on $X_1, X_2 \geq 0$ and $\sum_{i=1}^2 d_i X_i \leq 1.$ Therefore, if there is an $s_0$ such that $\frac{\varepsilon}{2} \mathcal{L}^2(s_0) > \max \left\lbrace \frac{1}{d_1}, \frac{1}{d_2} \right\rbrace $ then $\mathcal{L}^{'}(s_0) < 0.$ This yields $\mathcal{L}(s) \geq \mathcal{L}(s_0)$ for all $s \leq s_0,$ which contradicts $\lim_{s \to - \infty} \mathcal{L} =0.$

To prove the second statement, suppose that $\frac{\varepsilon}{2} \mathcal{L}^2(s_0) < \frac{1}{n}.$ Since $\sum_{i=1}^2 d_i X_i = 1$ in the Einstein locus, one has $\sum_{i=1}^2 d_i X_i^2 \geq \frac{1}{n}$ and therefore $\mathcal{L}^{'}(s_0) > 0.$ Hence, $\frac{\varepsilon}{2}\mathcal{L}^2$ is monotonically increasing whenever it is less than $\frac{1}{n}.$ 
\end{proof}

\begin{corollary}
Suppose that $d_1>1$ and $\widehat{D}>0.$ Then the corresponding expanding Ricci solitons and Einstein metrics with negative scalar curvature are complete.
\label{CompletenessEpsPosTwoSummands}
\end{corollary}
\begin{proof}
Suppose for contradiction that $t_{\max} < \infty.$ Due to \eqref{TimeRescaling} this is equivalent to saying that $\Vert \mathcal{L} \Vert_{L^1(0, \infty)} < \infty.$ However, since $\mathcal{L}$ is bounded due to lemma \ref{LemmaLBoundedAwayFromZero}, this implies $\mathcal{L} \in L^2(0, \infty).$ 
Hence the ODE for $\mathcal{L}$ yields $\mathcal{L}(s) \geq \mathcal{L}(0) \exp \left( -  \frac{\varepsilon}{2} \Vert \mathcal{L} \Vert_{L^2(0, \infty)} \right)  > 0$ and $\mathcal{L}$ is bounded away from zero for $s \geq 0.$ However, this contradicts $\mathcal{L} \in L^1(0, \infty).$
\end{proof}

\begin{remarkroman}
The Einstein metrics of negative scalar curvature in corollary \ref{CompletenessEpsPosTwoSummands} have already been constructed by B\"ohm \cite{BohmNonCompactEinstein} by different means, see remark \ref{RemarkConvergenceConeSolutions}.
\end{remarkroman}

\subsection{Ricci solitons from circle bundles}
\label{SectionSolitonsFromCircleBundles}

The two summands case allows the possibility $d_1 = 1$ and $A_1=0.$ Geometrically this case is realised by manifolds which are foliated by principal circle bundles over a Fano K\"ahler-Einstein manifold $(V,J,g).$ In this setting, examples of {\em K\"ahler} Ricci solitons have been found by Cao-Koiso \cite{CaoSoliton},\cite{KoisoSoliton} and Feldman-Ilmanen-Knopf \cite{FIKSolitons}. {\em Non-K\"ahler} examples have also been constructed independently by Stolarski \cite{StolarskiSteadyRSOnCxLineBundles} and Appleton \cite{AppletonSteadyRS}.

The precise geometric set-up is as follows: Recall that due to a theorem of Kobayashi \cite{KobayashiCompactFanos} any Fano manifold $V$ is simply connected and hence $H^2(V, \mathbb{Z})$ is torsion free. Therefore the first Chern class is $c_1(V, J)= p \rho$ for a positive integer $p$ and an indivisible class $\rho \in H^2(V, \mathbb{Z}).$ Suppose that the Ricci curvature of $(V,g)$ is normalised to be $\Ric = p g.$ If $\pi \colon P \to V$ is the principal circle bundle with Euler class $q \pi^{*} \rho$ for a non-zero integer $q \in \Z \setminus \left\{ 0 \right\}$ and $\theta$ the principal $S^1$-connection with curvature form $\Omega = q \pi^{*} \eta,$ where $\eta$ is the K\"ahler form associated to $g,$ then the Ricci soliton equation on $I \times P$ corresponding to the metric 
\begin{equation*}
dt^2 + f_1^2(t) \theta \otimes \theta + f_2^2(t) \pi^{*} g
\end{equation*}
is described by the two summands system with $d_1=1,$ $d_2=d=\dim_{\R} V$ and $A_1=0,$ $A_2=d_2 p,$ $A_3 = \frac{d_2 q^2}{4}.$ Notice also that the structure of the ODE has changed since $A_1=0.$ If the smoothness conditions \eqref{SmoothnessMetricTwoSummandsGeometricSetUpSection} are satisfied, this construction induces a smooth metric on the associated complex line bundle over $V.$ 

Metrics whose curvature tensor is invariant under the complex structure are considered by Dancer-Wang \cite{DWCohomOneSolitons} in the Ricci soliton case and by Wang-Wang \cite{WWEinsteinS2Bundles} in the Einstein case. This condition is equivalent to saying that
\begin{align*}
\frac{\dot{f}_2^2}{f_2^2} - \frac{q^2}{4} \frac{f_1^2}{f_2^4} = \left( - \dot{u} + \tr(L) + \frac{\dot{f}_1}{f_1} \right) \frac{\dot{f}_2}{f_2} - \frac{p}{f_2^2} + \frac{\varepsilon}{2}.
\end{align*}
As a special case, the {\em K\"ahler} condition reads
\begin{equation*}
\frac{\dot{f}_2}{f_2}= - \frac{q}{2} \frac{f_1}{f_2^2}
\end{equation*}
and it is preserved by the flow. In both cases, the equations can actually be integrated {\em explicitly.} In order to investigate {\em non-K\"ahler} trajectories, the Ricci soliton ODE will be studied qualitatively as before. To adjust the argument in proposition \ref{X2VariablePositive} to the conditions $d_1=1$ and $A_1=0,$ adopt the convention $\frac{A_1}{d_1(d_1-1)} = 1.$ That is, consider 
\begin{align*}
\widehat{\mathcal{G}}( \omega ) = \frac{1}{2} - \frac{p}{2} \omega^2 + \frac{d+2}{16}q^2 \omega^4 \ \text{ and } \
\mathcal{K} = \frac{1}{2} \left( \frac{X_1-X_2}{Y_1} \right) ^2 - \widehat{\mathcal{G}}( \omega )
\end{align*}
and note that $\widehat{\mathcal{G}}$ has two positive roots $0<\hat{\omega}_1 < \hat{\omega}_2$ if $2p^2 > (d+2)q^2.$ Then the proof of proposition \ref{X2VariablePositive} shows 

\begin{proposition} Suppose that $d_1=1,$ $A_1=0$ and $2p^2 > (d+2)q^2 >0.$ If $\varepsilon \geq 0,$ the set
\begin{align*}
\left\lbrace X_2 > 0 \ \text{ and } \ 0 < \frac{Y_2}{Y_1} <\hat{\omega}_1 \right\rbrace
\end{align*}
contains any trajectory of the Ricci soliton ODE that emanates from \eqref{InitialCriticalPoint} and flows into either \eqref{EinsteinLocus} or \eqref{SolitonLocus}. 
\label{X2PositiveCircleBundles}
\end{proposition}

Completeness of the metric can then be established as in proposition \ref{CompletenessEpsZeroTwoSummands} and corollary \ref{CompletenessEpsPosTwoSummands}. Notice in particular that long time existence still follows from corollary \ref{FlowExistsForAllTimes}. As the proof shows, even though $Y_1$ is not controlled by the conservation law \eqref{GeneralTwoSummandsConsLaw} anymore since $A_1 =0,$ it cannot blow up in finite time.

\begin{corollary} Let $d_1=1,$ $A_1 = 0,$ $2p^2 > (d+2)q^2 >0$ and $\varepsilon \geq 0.$ Then any trajectory of the Ricci soliton ODE which emanates from the critical point \eqref{InitialCriticalPoint} and lies in the Einstein locus \eqref{EinsteinLocus} or Ricci soliton locus \eqref{SolitonLocus} corresponds to a complete Einstein or Ricci soliton metric, respectively.
\end{corollary}

\begin{remarkroman}
(a) Notice on the contrary that the construction of K\"ahler Ricci solitons due to Feldman-Ilmanen-Knopf \cite{FIKSolitons} requires the condition $-q=p$ in the steady case and $-q > p$ in the expanding case, see also \cite[Theorem 4.20 and Remark 4.21]{DWCohomOneSolitons}. For example, in the case of $\mathbb{C}P^n$ one has $p=\frac{d+2}{2}$ and one thus requires $p>q^2>0$ for the argument of proposition \ref{X2PositiveCircleBundles} to work. In particular, the K\"ahler examples due to Feldman-Ilmanen-Knopf are not covered by the corollary. In the case of $\mathbb{C}P^n,$ these K\"ahler Ricci soliton metrics have also been investigated by Chave-Valent \cite{CVQasuiEinsteinRenormalization}.

(b) Explicit K\"ahler and non-K\"ahler Einstein metrics have already been described by Calabi \cite{CalabiKaehlerMetrics}, B{\'e}rard-Bergery \cite{BerardBergerySurDeNouvellesEintein}, Page-Pope \cite{PagePopeEinsteinMetricsOnCxLineBundles} and Wang-Wang \cite{WWEinsteinS2Bundles}.

(c) If $q > p,$ Appleton \cite{AppletonSteadyRS} proves that there cannot exist a complete Ricci flat metric on the associated complex line bundle. In particular, the corresponding trajectory cannot satisfy the bound $ \omega < \frac{4p}{(d+2)q^2}$ for all times.
\label{RemarkCircleBundleSolutions}
\end{remarkroman}

Observe that the initial stationary point \eqref{InitialCriticalPoint} is not hyperbolic in the case $d_1=1.$ Therefore a center manifold exists and the analysis before proposition \ref{NumberOfParameterFamilies} does not carry over. However, the work of Buzano \cite{BuzanoInitialValueSolitons} and Eschenburg-Wang \cite{EWInitialValueEinstein} still applies and the existence of Ricci soliton trajectories can be deduced, see also \cite{StolarskiSteadyRSOnCxLineBundles} or \cite{AppletonSteadyRS} for different arguments. Thus Theorem \ref{MainTheoremRSOnLineBundles} follows from the following result:

\begin{theorem} 
Suppose that $d_1=1,$ $A_1=0$ and $2p^2 > (d+2)q^2 >0.$ 

If $\varepsilon = 0$ there exists a $1$-parameter family and if $\varepsilon > 0$ a $2$-parameter family of trajectories lying in both the unstable manifold of \eqref{InitialCriticalPoint} and the Ricci soliton locus \eqref{SolitonLocus}. In particular, these give rise to complete Ricci soliton metrics on the total spaces of the corresponding complex line bundles over Fano K\"ahler-Einstein manifolds. 

Similarly, there exist a (up to homotheties) unique complete Ricci flat metric and a $1$-parameter family of complete Einstein metrics with negative scalar curvature on these spaces. \label{SolitonsOnCircleBundles}
\end{theorem}

It follows from the work of Appleton \cite{AppletonSteadyRS} that the Ricci solitons in theorem \ref{SolitonsOnCircleBundles} are asymptotically conical. Furthermore, recall from remark \ref{RemarkCircleBundleSolutions} that the existence of Einstein metrics is well known. 

\subsection{Asymptotics}
\label{SectionTwoSummandsAsymptotics}

This section discusses the asymptotic behaviour of the metrics which were constructed in section \ref{CompletenessTwoSummands}. In particular it will be shown that the steady Ricci solitons are asymptotically paraboloid and the expanding Ricci solitons are asymptotically conical. 

\subsubsection{Cone solutions}
\label{SectionConeSolutions} The concrete asymptotics of the metrics depend on the following well known construction, cf. \cite{BohmNonCompactEinstein} or \cite{DHWShrinkingSolitons}.

\begin{proposition}
Let $(P, g_E)$ be a homogeneous space with $Ric = (n-1) g_E.$ Then the metrics 
\begin{align*}
dt^2 & + \sin^2(t) g_E  \ \text{ for } \ t \in (0, \pi), \\ 
dt^2 +  t^2  g_E & \  \text{ and } \ dt^2 + \sinh^2(t) g_E   \ \text{ for } \ t > 0
\end{align*}
define cohomogeneity one Einstein metrics on $(0, \frac{\pi}{2}) \times P$ and $(0, \infty) \times P$ with Einstein constant $- \frac{\varepsilon}{2} = n, 0, -n,$ respectively. Any of these solutions will be called a {\em cone solution.}

Furthermore, the Ricci flat metrics together with the soliton potential $- \dot{u}(t) = \frac{\varepsilon}{2}t$ induce a shrinking or expanding Ricci soliton on $(0, \infty) \times P$ depending on whether $\varepsilon < 0$ or $\varepsilon >0.$ If $\varepsilon < 0$ these solutions are called {\em conical Gaussians.}
\label{ExplicitConeSolutionsProposition}
\end{proposition}

The above metrics have conical singularities at the singular orbits unless each singular orbit consists of a point. In this case the metrics correspond to the standard metrics on $S^{n+1},$ $\R^{n+1}$ and $\mathbb{H}^{n+1},$ respectively. To obtain concrete formulae in the two summands case, the following definitions are required.

\begin{definition}
Positive solutions $(c_1, c_2)$ to the equations
\begin{align}
(n-1) d_1 = \frac{A_1}{c_1^2} + A_3 \frac{c_1^2}{c_2^4}  \ \text{ and } \
(n-1) d_2 = \frac{A_2}{c_2^2} - 2 A_3 \frac{c_1^2}{c_2^4} \label{DefConeSolutionOne}
\end{align}
are called {\em cone solutions}.
\label{DefinitionConeSolutions}
\end{definition}

\begin{remarkroman} If $A_3 >0,$ the cone solutions take the explicit form
\begin{align*}
c_1^2 & = \frac{1}{2d_1+d_2} \left( \frac{A_2^2 d_1 + 4 A_1 A_3 (2d_1+d_2)}{2 A_3 (n-1)(2d_1+d_2)} \mp \sqrt{D} \right), \\
c_2^2 & = \frac{1}{2d_1+d_2} \left( A_2 n \pm 2 A_3 (2d_1+d_2) \sqrt{D} \right),
\end{align*}
where the discriminant $D$ is given by
\begin{equation}
D = \left( \frac{A_2}{2 A_3} \frac{d_1}{2d_1+d_2} \right)^2  - \frac{A_1}{A_3} \frac{d_2}{2d_1+d_2}.
\label{DiskriminatConeSolutions}
\end{equation}
Inserting the geometric definitions of the constants $A_1,$ $A_2,$ $A_3$ into \eqref{DiskriminatConeSolutions}, one obtains
\begin{align*}
D \geq 0 \ \text{ if and only if } \ \frac{(\Ric^{G/H})^2}{4 ||A||^2} \geq (2d_1+d_2) \frac{d_1-1}{d_1}.
\end{align*} 
Suppose that there are two real cone solutions. For a cone solution $(c_1, c_2),$ set $\omega=\frac{c_1}{c_2}.$ Then the ordering $\omega_1 < \omega_2$ defines the {\em first} and {\em second} cone solution. 

In particular, if $\widehat{D}>0,$ cf. \eqref{DefinitionDHat}, there exist two cone solutions and it is easy to check that $\omega_1 < \hat{\omega}_1 < \omega_2 < \hat{\omega}_2$ in this case. This has also been observed by B\"ohm \cite{BohmInhomEinstein}.
\label{ExplicitFormulaConeSolutions}
\end{remarkroman}

Let $D \geq 0$ and let $(c_1, c_2)$ be a cone solution as in definition \ref{DefinitionConeSolutions}. With the normalisation of the Einstein constant $- \frac{\varepsilon}{2} \in \left\{ -n, 0, n \right\},$ the two summands Einstein cone solutions of proposition \ref{ExplicitConeSolutionsProposition} take the form 
\begin{align}
f_i(t)= c_i \sin(t) \ \text{ for } \ t \in (0, \pi)
\label{ExplicitConeSolutionScalPos}
\end{align}
in the case of positive scalar curvature and 
\begin{align}
f_i(t)= c_i t \ \text{ and } \ f_i(t)= c_i \sinh(t) \ \text{ for } \ t > 0
\label{ExplicitConeSolutionScalNonPos}
\end{align}
in the Ricci flat and negative scalar curvature case, respectively. Any cone solution is called {\em first} cone solution if that is the case for the pair $(c_1,c_2)$ as remark \ref{ExplicitFormulaConeSolutions}.

\begin{example} \normalfont Recall the examples of group diagrams in table \ref{HopfFibrationsTable}, which induce the Hopf fibrations. In the $\mathbb{H}P^{m+1}$-example the cone solutions are 
\begin{align*}
c_1^2 & = \frac{9+14m+4m^2}{(1+2m)(3+2m)^2} \ \text{ and } \ c_2^2 = \frac{9+14m+4m^2}{(1+2m)(3+2m)}, \\
c_1^2 & = c_2^2 = 1,
\end{align*}
in the $F^{m+1}$-example they are given by
\begin{align*}
c_1^2 & = \frac{(1+m)^2+m}{(1+m)^2(1+4m)} \ \text{ and } \ c_2^2 = 4 \frac{(1+m)^2 + m}{(2m+1)^2+m}, \\
c_1^2 & = \frac{1+m}{1+4m} \ \text{ and } \ c_2^2 = 4 c_1^2,
\end{align*}
and in the $CaP^2$-example they are $c_1^2 = \frac{57}{121},$ $c_2^2 = \frac{19}{11}$ and $c_1^2 = c_2^2 = 1.$ In all cases, the first pair also describes the first cone solution. 
\label{ExamplesConeSolutionsFromHopfFibrations}
\end{example}

The following elementary but useful characterisation of $\omega_1$ and $\omega_2$ is immediate from definition \ref{DefinitionConeSolutions} and remark \ref{ExplicitFormulaConeSolutions}.

\begin{proposition}
Let $D > 0.$ Then the two positive roots of the function 
\begin{equation}
f( \omega ) = \frac{A_1}{d_1} - \frac{A_2}{d_2} \omega^2 + A_3 \left( \frac{1}{d_1} + \frac{2}{d_2} \right)  \omega^4
\label{FunctionDeterminingSecondDerivativeOmega}
\end{equation}
are the ratios $\omega_1,$ $\omega_2$ of the first and second cone solution, respectively, i.e.
\begin{align*}
\omega_{1}^2 = \frac{A_2}{2A_3} \frac{d_1}{2d_1+d_2} - \sqrt{ D } \ \ \text{and} \ \ \omega_{2}^2 = \frac{A_2}{2A_3} \frac{d_1}{2d_1+d_2} + \sqrt{ D }.
\end{align*}
In particular, it follows that $\omega_1^2 < \frac{A_2}{4A_3}$ and $\omega_2^2 < \frac{A_2}{2A_3}.$
\label{CharacterisationOfConeSolutionRatio}
\end{proposition}

\subsubsection{Steady Ricci solitons}
\label{SteadyRSAymptoticsSection}

The rotationally symmetric Bryant soliton on $\R^n,$ $n \geq 3,$ is asymptotically paraboloid and therefore non-collapsed. It will be shown that this is also the case for the non-trivial steady Ricci solitons constructed in section \ref{CompletenessTwoSummands}. 

\vspace{2mm}

Recall from proposition \ref{CompleteSteadyRSAsymptotics} that on a complete, non-trivial cohomogeneity one steady Ricci soliton there holds $-\dot{u}(t) \to \sqrt{-C}$ as $t \to \infty$ and $0 < \tr(L) \leq \frac{n}{t}$ for $t >0.$ Therefore, if the shape operator remains positive definite, it follows that $\frac{\dot{f}_i}{f_i} \to 0$ as $t \to \infty.$ According to corollary \ref{FlowExistsForAllTimes}, this automatically holds in the two summands case if $d_1 > 1$ and $\widehat{D}>0.$

In order to obtain the concrete asymptotics of the metric if $A_3 >0$, an understanding of the long time behaviour of $\omega$ is essential:

\begin{proposition}
Let $d_1 >1,$ $\widehat{D}>0$ and $\varepsilon = 0.$ Then along trajectories of non-trivial steady Ricci solitons the limit $\omega_{\infty} = \lim_{t \to \infty} \omega(t)$ exists and $\lim_{t \to \infty} \dot{\omega}(t) = 0.$ 
\label{OmegaConverges}
\end{proposition}
\begin{proof}
Let $v(t) = \sqrt{\det g_t} = f_1^{d_1}(t) f_2^{d_2}(t)$ denote the relative volume of the principal orbit and consider the variables $\frac{v^{1/n}}{f_i}$ and $\frac{v^{2/n} \dot{f}_i}{f_i}$ for $i=1,2.$ Observe that $\frac{v^{1/n}}{f_1} = \frac{1}{\omega^{d_2/n}}$ and $\frac{v^{1/n}}{f_2} = \omega^{d_1/n}.$ 

Therefore, the B\"ohm functional has the lower bound
\begin{align*}
\mathscr{F}_0 = v ^{\frac{2}{n}} \left( \tr(r_t) + \tr(( L^{(0)})^2 )\right) 
\geq  v ^{\frac{2}{n}} \tr(r_t) 
= \frac{A_1}{\omega^{2 d_2/n}} + A_2 \omega^{2 d_1/n} - A_3 w^{2(2 d_1 + d_2)/n}.
\end{align*}
Since $\mathscr{F}_0$ is non-increasing, $v ^{\frac{2}{n}} \tr(r_t)$ is bounded from above for $t \geq t_0 > 0$ and hence $\omega$ is bounded away from zero for these $t.$ As $\omega < \hat{\omega}_1,$ the variables $\frac{v^{1/n}}{f_i}$ are hence bounded for $t \geq t_0.$ 

Furthermore, the variables $\frac{v^{2/n} \dot{f}_i}{f_i}$ satisfy the ODE system
\begin{align*}
\frac{d}{dt} \frac{v^{2/n} \dot{f}_1}{f_1} & = - ( - \dot{u} + \frac{n-2}{n} \tr(L) ) \frac{v^{2/n} \dot{f}_1}{f_1} 
+ \frac{A_1}{d_1} \left( \frac{v^{1/n}}{f_1} \right)^2 + \frac{A_3}{d_1}  \omega^2 \left( \frac{v^{1/n}}{f_2} \right)^2, \\
\frac{d}{dt} \frac{v^{2/n} \dot{f}_2}{f_2} & = - ( - \dot{u} + \frac{n-2}{n} \tr(L) ) \frac{v^{2/n} \dot{f}_2}{f_2} 
+ \frac{A_2}{d_2} \left( \frac{v^{1/n}}{f_2} \right)^2 -  \frac{2 A_3}{d_2}  \omega^2 \left( \frac{v^{1/n}}{f_2} \right)^2.
\end{align*}
Due to the known asymptotics, the coefficient of $\frac{v^{2/n} \dot{f}_i}{f_i}$ tends to $- \sqrt{-C}$ and the remaining polynomial terms are bounded. Hence, by comparison, the variables $\frac{v^{2/n} \dot{f}_i}{f_i}$ remain bounded. Therefore, one can pass to the $\omega$-limit set $\Omega.$ Due to its monotonicity, $\mathscr{F}_0$ converges. Its derivative \eqref{DerivativeOfBohmFunctional} has to vanish on $\Omega$ and therefore the limiting value $\mathscr{F}_0 = (v ^{\frac{2}{n}} \tr(r))_{\infty}$ can be expressed in terms of $\omega$ as above. In particular, $\omega$ converges. 

The asymptotics of $\dot{\omega}$ simply follow from the ODE $\dot{\omega} = \omega \left\lbrace \frac{\dot{f}_1}{f_1} - \frac{\dot{f}_2}{f_2} \right\rbrace$ and the fact that $\frac{\dot{f}_i}{f_i} \to 0$ as $t \to \infty.$ 
\end{proof}

\begin{remarkroman}
It is also possible to derive an integral formula for $\dot{\omega}.$ Indeed, it is straightforward to check that
\begin{align*}
\frac{d}{dt} \left\lbrace \dot{\omega} e^{-u} f_1^{d_1-1} f_2^{d_2+1} \right\rbrace = f(\omega) e^{-u} f_1^{d_1-2} f_2^{d_2} ,
\end{align*}
where $f(\omega)$ is defined in \eqref{FunctionDeterminingSecondDerivativeOmega}. If $d_1 > 1,$ it follows that 
\begin{align*}
\dot{\omega}(t) = \frac{e^u(t)}{f_1^{d_1-1}(t) f_2^{d_2+1}(t)} \cdot \int_{0}^{t} f( \omega(s) ) e^{-u(s)} f_1^{d_1-2}(s) f_2^{d_2}(s) ds.
\end{align*}
Since $f_1, f_2$ are monotonic and $e^{u(t)} \int_{0}^t e^{-u(s)} ds \to \frac{1}{\sqrt{-C}}$ as $t \to \infty$ due to L'H\^{o}pital's rule, one has the bound $\dot{\omega}(t) \leq  \overline{C} \cdot \frac{1}{f_1(t) f_2(t)}$ for some constant $\overline{C} > 0.$
\label{IntegralFormulaOmegaDot}
\end{remarkroman}

Now the asymptotics of the metric can be deduced:

\begin{proposition}
Let $d_1>1,$ $A_1>0, \widehat{D}>0$ and suppose that $(d_1+1) \ddot{u}(0) = C < 0.$ Then the corresponding two summands steady Ricci soliton metrics satisfy
\begin{align*}
- \dot{u}(t) \to \sqrt{-C} \ \text{ and } \ \frac{f_i^2(t)}{t} \to \frac{2}{\sqrt{-C}}(n-1) c_i^2
\end{align*}
as $t \to \infty,$ where $(c_1, c_2)$ denotes the first cone solution. In particular, $\omega \to \omega_1$ as $t \to \infty.$
\label{SummarisingAsymptoticsSteadyRicciSoliton}
\end{proposition}
\begin{proof}
Recall that $- \dot{u}(t) \to \sqrt{-C}$ as $t \to \infty$ due to proposition \ref{CompleteSteadyRSAsymptotics}. Notice that $f_1,$ $f_2$ satisfy
\begin{align*}
\ddot{f}_1 & = - ( - \dot{u} - d_2 \frac{\dot{\omega}}{\omega} ) \dot{f}_1 - (n-1) \frac{\dot{f}_1^2}{f_1} + \frac{A_1+A_3 \omega^4}{d_1 f_1}, \\ 
\ddot{f}_2 & = - ( - \dot{u} + d_1 \frac{\dot{\omega}}{\omega} ) \dot{f}_2 - (n-1) \frac{\dot{f}_2^2}{f_2} + \frac{A_2-2 A_3 \omega^2}{d_2 f_2}.
\end{align*}

As $\omega < \hat{\omega}_1 < \sqrt{\frac{A_2}{2 A_3}},$ $A_1 > 0$ and $\omega$ converges, both $f_1, f_2$ satisfy a differential equation of the form
\begin{align*}
\ddot{f} = - a_1 \dot{f} - (n-1) \frac{\dot{f}^2}{f} + \frac{a_2}{2 f},
\end{align*}
where $a_i \colon [0,\infty) \to \R$ are smooth functions with $\lim_{t \to \infty} a_i(t) = a_i^{\ast} >0.$ Set $A= \min \lbrace a_1^{\ast}, a_2^{\ast}\rbrace.$ It is shown in \cite[Lemma 6.2]{AppletonSteadyRS} that for every $\varepsilon \in (0, A)$ and every solution $f: [0, \infty) \to \R$ with $f(0), \dot{f}(0)>0$ there exists $t_0>0$ such that 
\begin{align*}
f(t_0)^2 + \gamma_{-} \left(1+ \varepsilon \right)^{-1} \left( t-t_0 \right) \leq f^2(t) \leq f(t_0)^2 + \gamma_{+} \left(t-t_0 \right)
\end{align*}
for all $t >t_0$, where $\gamma_{\pm} = \frac{a_2^{\ast} \pm \varepsilon}{a_1^{\ast} \mp \varepsilon}$.

It follows that $\frac{\gamma_{1,{-}}}{\gamma_{2,{+}}} \leq \omega_{\infty}^2 \leq \frac{\gamma_{1,{+}}}{\gamma_{2,{-}}}$ for every sufficiently small $\varepsilon >0.$ In the limit as $\varepsilon \to 0$ one obtains equality and thus
\begin{align*}
\omega_{\infty}^2 =\frac{d_2}{d_1} \frac{A_1+A_3 \omega_{\infty}^4}{A_2 - 2 A_3 \omega_{\infty}^2}.
\end{align*}
In particular, $0 < \omega_{\infty} \leq \hat{\omega}_1$ is a root of $f(\omega)$ and due the characterisation \ref{CharacterisationOfConeSolutionRatio} of the cone solutions, it follows that $\omega_{\infty} = \omega_1$ is the ratio of the first cone solution.

The asymptotic behaviour of $f_1,$ $f_2$ now follows with the formulae in definition \ref{DefinitionConeSolutions}. 
\end{proof}

\begin{remarkroman}
For $d_1 >1,$ $\widehat{D}>0$ and $\varepsilon = 0,$ the asymptotics of the rescaled Ricci soliton ODE of section \ref{CompletenessTwoSummands} are 
\begin{align*}
X_1, X_2 \to 0 \ \text{ and } \ Y_1, Y_2 \to 0 \ \text{ and } \ \mathcal{L} \to \frac{1}{\sqrt{-C}}
\end{align*}
as $s \to \infty.$
\label{VariablesBoundedSteadySolitonTwoSummandsCase}
\end{remarkroman}

\subsubsection{Expanding Ricci solitons}
\label{ExpandingRSAymptoticsSection}

It will be shown that the expanding Ricci solitons are asymptotically conical at infinity and the soliton potential grows quadratically at infinity. 

Recall from \eqref{GeneralAsymptoticsExpandingRS} that on a complete, non-trivial cohomogeneity one expanding Ricci soliton $-\dot{u}$ is asymptotically linear and the mean curvature of the principal orbit is bounded. Furthermore, corollary \ref{FlowExistsForAllTimes} implies that the shape operator is positive definite in the two summands case. The definition of the rescaled variables in \eqref{RescaledTwoSummandsVariables} thus implies: 

\begin{proposition}
Suppose that $d_1 >1$ and $\widehat{D} > 0$ and consider the flow of the Ricci soliton ODE in the phase space of expanding Ricci solitons. Then 
\begin{align*}
X_1, X_2 \to 0 \ \text{ and } \ Y_1, Y_2 \to 0 \ \text{ and } \ \mathcal{L} \to 0
\end{align*}
as $s \to \infty.$ 
\end{proposition}

A modification of the discussion in \cite{DWExpandingSolitons} can now be used to deduce the claimed asymptotically conical geometry at infinity:

\begin{proposition}
Suppose that $d_1 >1$ and $\widehat{D} > 0.$ Then along trajectories corresponding to non-trivial expanding Ricci soliton metrics the soliton potential and shape operator satisfy
\begin{align*}
\frac{- \dot{u}(t)}{t} \to \frac{\varepsilon}{2} \ \text{ and } \ t \cdot L_t \to \mathbb{I}_n
\end{align*}
as $t \to \infty.$
\label{AsymptoticsOfExpandingTwoSummandsRS}
\end{proposition}
\begin{proof}
Consider the ODE system
\begin{align*}
\frac{d}{ds} \frac{X_1}{\mathcal{L}^2} & = \left(- \sum_{i=1}^2 d_i X_i^2 -1 \right) \frac{X_1}{\mathcal{L}^2} + \frac{\varepsilon}{2} \left( 1 + X_1 \right) + \frac{A_1}{d_1} \left( \frac{Y_1}{\mathcal{L}} \right)^2 + \frac{A_3}{d_1} \left( \frac{Y_2}{\mathcal{L}} \right)^2 \left( \frac{Y_2}{Y_1} \right)^2, \\
\frac{d}{ds} \frac{X_2}{\mathcal{L}^2} & = \left(- \sum_{i=1}^2 d_i X_i^2 -1 \right) \frac{X_2}{\mathcal{L}^2} + \frac{\varepsilon}{2} \left( 1 + X_2 \right) + \frac{A_2}{d_2} \left( \frac{Y_2}{\mathcal{L}} \right)^2 - \frac{2 A_3}{d_2} \left( \frac{Y_2}{\mathcal{L}} \right)^2 \left( \frac{Y_2}{Y_1} \right)^2
\end{align*}
and notice that $\frac{d}{ds} \frac{Y_i}{\mathcal{L}} = - \frac{Y_i}{\mathcal{L}} X_i$ implies that both limits $\hat{y}_i = \lim_{s \to \infty} \frac{Y_i}{\mathcal{L}} \in [0, \infty)$ exist.
 
If $\hat{y}_1 = 0$ then, as $\omega = \frac{Y_2}{Y_1}$ remains bounded, one necessarily also has $\hat{y}_2 = 0.$ In this case one can proceed as in \cite[Lemma 3.15]{DWExpandingSolitons} to show that $\frac{X_i}{\mathcal{L}^2} \to \frac{\varepsilon}{2}$ as $s \to \infty$ because the extra terms involving $A_3$ tend to zero. Similarly, integrating the ODE for $\mathcal{L}$ implies $\mathcal{L}^2 \cdot s \to \frac{1}{\varepsilon}$ as $s \to \infty.$ Since $dt = \mathcal{L} ds,$ this yields $s \sim \frac{\varepsilon}{4} t^2$ and hence $\mathcal{L} \cdot t \to \frac{2}{\varepsilon}$ as $t \to \infty$. The claim then follows from the definition of the coordinate change in \eqref{RescaledTwoSummandsVariables}.

It remains to rule out that possibly $\hat{y}_1 > 0.$ In this case the existence of $\omega_{\infty} = \lim_{s \to \infty}  \frac{Y_2}{Y_1}$ is immediate. Hence the ODEs imply 
\begin{align*}
\frac{X_1}{\mathcal{L}^2} \to \frac{\varepsilon}{2} + \frac{1}{d_1} \left( A_1 \hat{y}_1^2+A_3 \hat{y}_2^2 \omega_{\infty}^2  \right) 
\ \text{ and } \ 
\frac{X_2}{\mathcal{L}^2} \to \frac{\varepsilon}{2} + \frac{1}{d_2} \left( A_2 \hat{y}_2^2 - 2 A_3 \hat{y}_2^2 \omega_{\infty}^2  \right)
\end{align*}
as $s \to \infty$ and both limits are positive as $\varepsilon > 0$ and $\omega_{\infty}^2 < \frac{A_2}{2 A_3}.$ Set $\Lambda_i = \lim_{s \to \infty} \frac{X_i}{\mathcal{L}^2} > 0.$ It follows that
\begin{align*}
\frac{Y_i^{'}}{\mathcal{L}^{'}} = \frac{Y_i \left(\sum_{i=1}^{2} d_i X_i^2 - \frac{\varepsilon}{2} \mathcal{L}^2 - X_i\right)}{\mathcal{L}\left(\sum_{i=1}^{2} d_i X_i^2 - \frac{\varepsilon}{2} \mathcal{L}^2 \right)} \to \hat{y}_i \cdot \frac{\varepsilon + 2 \Lambda_i}{\varepsilon}
\end{align*}
as $s \to \infty,$ but if $\hat{y}_1 > 0$ L'H\^{o}pital's rule implies $\Lambda_1=0$ and thus a contradiction. 
\end{proof}

\subsubsection{Ricci flat metrics}
\label{SubSectionRFMetrics}

The rescaled coordinates of section \ref{CompletenessTwoSummands} are particularly suited to analyse the Ricci flat trajectories. The induced Ricci flat metric is asymptotically conical and in fact is asymptotic to the first cone solution $f_i(t) = c_i t.$ This also follows from B\"ohm's \cite{BohmNonCompactEinstein} original construction, see remark \ref{RemarkConvergenceConeSolutions}. 
 
\begin{proposition}
Let $d_1 >1$ and $\widehat{D}>0.$ Along trajectories of the Ricci flat system $X_i \to \frac{1}{n}$ and $Y_i \to \frac{1}{n c_i}$ as $s \to \infty,$ where $(c_1,c_2)$ denotes the first cone solution.
\label{VariablesBoundedRicciFlatTwoSummandsCase}
\end{proposition}
\begin{proof}
Recall from corollary \ref{FlowExistsForAllTimes} that the variables $X_i, Y_i$ for $i=1,2$ are all positive and bounded along the flow since $d_1 >1$ and $\widehat{D}>0.$ To deduce the asymptotics, consider the function 
\begin{equation*}
\mathcal{G}= Y_1^{d_1} Y_2^{d_2},
\end{equation*}
which is in fact the inverse of B\"ohm's Lyapunov \eqref{BohmFunctional} in the $X$-$Y$-coordinates. Its derivative is given by 
\begin{equation*}
\mathcal{G}^{'} = n \mathcal{G} \left\{ \sum_{i=1}^2 d_i X_i^2 - \frac{1}{n} \right\}
\end{equation*}
and hence it is non-decreasing and bounded. Thus, it converges to a finite positive limit as $s \to \infty.$ This also shows that $Y_1, Y_2$ are bounded away from zero as $s \to \infty.$ Standard ODE theory now implies that the $\omega$-limit set $\Omega$ of the flow of $X_1, X_2,$ $Y_1, Y_2$ is non-empty, compact, connected and flow invariant. As $\mathcal{G}$ is monotonic and bounded, it must be constant on $\Omega.$ But since $d_1 X_1 + d_2 X_2 =1$ there holds $\mathcal{G}^{'} = 0$ if and only if $X_1 = X_2 = \frac{1}{n}.$ Moreover, this yields
\begin{align*}
0 & = X_1^{'} = \frac{1}{n} \left( \frac{1}{n} - 1 \right) + \frac{A_1}{d_1} Y_1^2 + \ \frac{A_3}{d_1} \frac{Y_2^4}{Y_1^2}, \\
0 & = X_2^{'} = \frac{1}{n} \left( \frac{1}{n} - 1 \right) + \frac{A_2}{d_2} Y_2^2 - 2 \frac{A_3}{d_2} \frac{Y_2^4}{Y_1^2}
\end{align*}
on the $\omega$-limit set. In particular, the pair $((n Y_1)^{-1}, (n Y_2)^{-1})$ satisfies the equations \eqref{DefConeSolutionOne} of the cone solutions. Since the bound $Y_2/Y_1 < \hat{\omega}_1$ holds along the flow and $\omega_1 < \hat{\omega}_1 < \omega_2 < \hat{\omega}_2,$ it follows that $Y_i \to \frac{1}{nc_i}$ as $s \to \infty,$ where $(c_1,c_2)$ describes the first cone solution. This completes the proof.
\end{proof}

The asymptotic behaviour of the metric can be deduced from $\dot{f}_i = \frac{X_i}{Y_i} \to c_i$ as $t \to \infty$. The metric is therefore asymptotically conical at infinity. 

\subsubsection{Ricci flat metrics: Explicit trajectories and rotational behaviour}
\label{RicciFlatPlanarSystem}

It is a special feature of the Ricci flat equation that it reduces to a planar system for the variables $X_1, Y_1$ as the variable $\mathcal{L}$ decouples completely. More generally, in the Einstein case there holds $X_2= \frac{1}{d_2}(1 - d_1 X_1)$ and the conservation law \eqref{GeneralTwoSummandsConsLaw} then determines $Y_2$ in terms of $X_1, Y_1$ and $\frac{\varepsilon}{2} \mathcal{L}^2.$ Explicitly, it is given by
\begin{equation*}
Y_2^2 = \frac{A_2}{2 A_3} Y_1^2 \pm \frac{1}{2 A_3} \sqrt{A_2^2 Y_1^4 + 4 A_3 \left( \sum_{i=1}^2 d_i X_i^2 +(n-1) \frac{\varepsilon}{2} \mathcal{L}^2 - 1 + A_1 Y_1^2 \right) Y_1^2 }.
\end{equation*}
Initially, $Y_2$ is given by the solution corresponding to '$-$' as $\lim_{s \to - \infty} Y_2 =0.$ Notice that the discriminant vanishes if and only if $Y_2^2/Y_1^2= \frac{A_2}{2 A_3}$ and recall that if $d_1 >1,$ $\widehat{D} > 0$ and $\varepsilon \geq 0$ the estimate $Y_2^2/Y_1^2 < \hat{\omega}_1^2 < \frac{A_2}{2 A_3}$ has been established. Hence, in this case, only the '$-$' solution is realised by the flow.

Therefore, consider the ODE system 
\begin{align*}
X_1^{'} = \ &  \left(  X_1 + \frac{1}{d_1} \right) \left( n \frac{d_1}{d_2}  X_1^2 - 2 \frac{d_1}{d_2} X_1 + \frac{1}{d_2} - \frac{\varepsilon}{2} \mathcal{L}^2 - 1 \right) \\
 & \ + \left( 2 A_1 + \frac{A_2^2}{2 A_3} \right) \frac{Y_1^2}{d_1}+ \frac{\varepsilon}{2}\left( 1 + \frac{n}{d_1} \right)  \mathcal{L}^2 \\ 
 & \ - \frac{A_2}{2 d_1 A_3} \sqrt{A_2^2 Y_1^4 + 4 A_3 \left( \sum_{i=1}^2 d_i X_i^2 +(n-1) \frac{\varepsilon}{2} \mathcal{L}^2 - 1 + A_1 Y_1^2 \right) Y_1^2 }, \\
Y_1^{'} = & \ Y_1 \left( n \frac{d_1}{d_2}  X_1^2 - 2 \frac{d_1}{d_2} X_1 + \frac{1}{d_2} - \frac{\varepsilon}{2} \mathcal{L}^2 - X_1 \right), \\
\mathcal{L}^{'} = & \ \mathcal{L} \ \left( n \frac{d_1}{d_2}  X_1^2 - 2 \frac{d_1}{d_2} X_1 + \frac{1}{d_2} - \frac{\varepsilon}{2} \mathcal{L}^2 \right).
\end{align*} 
In the Ricci flat case this yields indeed a $2$-dimensional system for $X_1$ and $Y_1.$ Moreover, one has $\mathcal{L}(s) = \mathcal{L}(s_0) \exp \left[ \int_{s_0}^s   \left( n \frac{d_1}{d_2}  X_1^2 - 2 \frac{d_1}{d_2} X_1 + \frac{1}{d_2} \right) d \tau \right].$

Recall from proposition \ref{VariablesBoundedRicciFlatTwoSummandsCase} that one expects $(X_1,Y_1) \to (\frac{1}{n}, \frac{1}{n c_1})$ as $s \to \infty$ if the cone solutions are real. To study the dynamics of the planar $(X_1, Y_1)$-system close to the stationary point $(\frac{1}{n},\frac{1}{n c_1}),$ consider its linearisation at that point. It is described by the matrix
\begin{equation*}
\begin{pmatrix}
-\frac{n-1}{n} & 2 \frac{c_1}{n}\left[ n-1 - 2 A_3 \frac{c_1^2}{c_2^4}\left( \frac{1}{d_1} + \frac{1}{d_2} \right) \right]  \\
-\frac{1}{c_1 n} & 0
\end{pmatrix}.
\end{equation*}
The eigenvalues are the solutions to the quadratic equation
\begin{equation*}
\lambda^2 + \frac{n-1}{n} \lambda + \frac{2}{n^2} \left[ n-1 - 2 A_3 \frac{c_1^2}{c_2^4}\left( \frac{1}{d_1} + \frac{1}{d_2} \right) \right] = 0
\end{equation*}
and it is therefore easy to deduce:

\begin{corollary} The limiting point of the Ricci flat trajectories is a stable spiral if and only if  
\begin{equation}
\frac{(n-1)(n-9)}{8} + 2 A_3 \frac{c_1^2}{c_2^4} \left( \frac{1}{d_1} + \frac{1}{d_2} \right) < 0.
\end{equation}
In particular, if $A_3 =0$ this is equivalent to $2 \leq n \leq 8.$ Otherwise, it is a stable node.
\label{StableSpiral}
\end{corollary}

The reduction to the planar $(X_1,Y_1)$-system can also be used to describe explicit trajectories. Trajectories which correspond to smooth complete Ricci flat metrics must emanate from $( \frac{1}{d_1}, \frac{1}{d_1} )$ and are expected to converge to $( \frac{1}{n}, \frac{1}{n c_1} ).$ In low dimensional examples, these trajectories are actually realised by straight lines! This can be seen by introducing polar coordinates centred at $( \frac{1}{d_1}, \frac{1}{d_1} )$, and a straightforward calculation verifies that the angle remains constant.

This provides a new coordinate representation of metrics of special holonomy considered by Bryant-Salamon  \cite{BSExceptionalHolonomy} and Gibbons-Page-Pope \cite{GPPEinsteinOnSphereR3R4bundles}.

\begin{theorem}
On the open disc bundles associated to the group diagrams $G=Sp(2),$ $H=Sp(1) \times Sp(1),$ $K= U(1) \times Sp(1)$ and $G=Sp(1) \times Sp(2),$ $H=Sp(1) \times Sp(1) \times Sp(1),$ $K=Sp(1) \times Sp(1)$ the trajectories of the complete Ricci flat two summands metrics are line segments when represented in the above coordinate system.
\label{ExplicitRFTrajectories}
\end{theorem}

\subsubsection{Einstein metrics with negative scalar curvature}
\label{SubSectionAsymptoticsEinsteinScalNeg}
It will be shown that in this case the B\"ohm functional $\mathscr{F}_0$ asymptotically approaches the value of the first cone solution, and hence work of B\"ohm implies that the metric is in fact asymptotic to the first cone solution $f_i(t) = c_i \sinh(t)$.

\begin{proposition} Let $d_1 >1$ and $\widehat{D}>0.$ Then the asymptotic behaviour of trajectories corresponding to  complete Einstein metrics with negative scalar curvature is given by
\begin{align*}
X_1, X_2 \to \frac{1}{n} \ \text{ and } \ Y_1, Y_2 \to 0, \ \omega = \frac{Y_2}{Y_1} \to \omega_1 \ \text{ and } \ \mathcal{L} \to \sqrt{\frac{2}{n \varepsilon}}
\end{align*}
as $s \to \infty,$ where $\omega_1 = \frac{c_1}{c_2}$ is the ratio of the first cone solution. 

Furthermore, $\mathscr{F}_0 \to n(n-1) c_{1}^{{2d_1}/{n}} c_{2}^{{2d_2}/{n}}$ as $s \to \infty,$ which is the value of $\mathscr{F}_0$ evaluated on the first cone solution $(c_1, c_2).$
\label{TwoSummandsEinsteinMetricsWithNegativeScalarCurvature}
\end{proposition}
\begin{proof}
As in the proof of corollary \ref{FlowExistsForAllTimes}, introduce the variable $\omega = \frac{Y_2}{Y_1}$ in order to view the Ricci soliton equation as an ODE with polynomial right hand side. Furthermore, all variables remain bounded along the flow and hence the $\omega$-limit set $\Omega$ is non-empty, connected, compact and flow-invariant.

Recall from lemma \ref{LemmaLBoundedAwayFromZero} that $\mathcal{L}(s)$ is bounded away from zero for $s \geq 0.$ As the quotients $\frac{Y_i}{\mathcal{L}}$ satisfy $\frac{d}{ds} \frac{Y_i}{\mathcal{L}} = - \frac{Y_i}{\mathcal{L}} X_i,$ they are monotonically decreasing and hence converge as $s \to \infty.$ Moreover, the quotients are well-defined on $\Omega.$ Therefore their derivatives vanish, which implies $Y_i \cdot X_i = 0$ on $\Omega.$ But due to the bounds on $\mathcal{L}$ in lemma \ref{LemmaLBoundedAwayFromZero}, $X_1$ is bounded away from zero and in particular is non-zero on $\Omega.$ This implies $0 < Y_2 < \hat{\omega}_1 Y_1 \to 0$ as $s \to \infty.$ 

Now consider the evolution of the B\"ohm functional $\mathscr{F}_0 = v ^{\frac{2}{n}} \left( \tr(r_t) + \tr(( L^{(0)})^2 )\right),$ which was introduced in \eqref{BohmFunctional}. In the current coordinate system it is given by
\begin{align*}
\mathscr{F}_0 & 
= \prod_{i=1}^2 Y_i^{-2d_i / n} \left\lbrace \sum_{i=1}^2 A_i Y_i^2 - A_3 \frac{Y_2^4}{Y_1^2} + \sum_{i=1}^2 d_i X_i^2 - \frac{1}{n}\left( \sum_{i=1}^2 d_i X_i \right) ^2 \right\rbrace \\
& = \prod_{i=1}^2 Y_i^{-2d_i / n} \left\lbrace A_1 Y_1^2 + Y_2^2 \left( A_2 - A_3 \frac{Y_2^2}{Y_1^2} \right) + \sum_{i=1}^2 d_i X_i^2 - \frac{1}{n} \right\rbrace.
\end{align*}
Observe that it is bounded from below by zero as $\frac{Y_2}{Y_1} = \omega < \hat{\omega}_1 < \frac{A_2}{2 A_3}.$ Furthermore, according to \eqref{DerivativeOfBohmFunctional}, $\mathscr{F}_0$ is non-increasing and therefore converges as $s \to \infty.$ However, for $\mathscr{F}_0$ to be finite on the $\omega$-limit set $\Omega,$ one has to have $\sum_{i=1}^2 d_i X_i^2 = \frac{1}{n},$ which forces $X_1 = X_2 = \frac{1}{n}$ in the Einstein locus $\sum_{i=1}^2 d_i X_i = 1$ as $X_1, X_2 \geq 0.$ Therefore $X_1$ is constant on $\Omega$ and then also $\mathcal{L}$ due to the ODE for $X_1.$ Finally, the ODE for $\mathcal{L}$ itself shows that $\frac{\varepsilon}{2} \mathcal{L}^2 = \frac{1}{n}$ on $\Omega.$

To deduce the asymptotic behaviour of $\omega$, first observe that the monotonicity of $\mathscr{F}_0$ and
\begin{align*}
\mathscr{F}_0 & = \frac{1}{\omega^{2 d_2 / n}} \left( A_1 + A_2 \omega^2 - A_3 \omega^4 \right) + v^{2/n} \tr( ( L^{(0)})^2) \\
& = \frac{A_1}{\omega^{2 d_2/n}} + A_2 \omega^{2 d_1/n} - A_3 w^{2(2 d_1 + d_2)/n} + v^{2/n} \tr( ( L^{(0)})^2)
\end{align*}
imply that $\omega$ is bounded away from zero for $t \geq t_0 > 0.$ Notice furthermore that
\begin{align*}
\frac{d}{dt} v^{2/n} \tr( ( L^{(0)})^2) & = \frac{d}{dt} \mathscr{F}_0 - \frac{d}{dt} \left\lbrace \frac{A_1}{\omega^{2 d_2/n}} + A_2 \omega^{2 d_1/n} - A_3 w^{2(2 d_1 + d_2)/n} \right\rbrace  \\ 
& = -2 \frac{n-1}{n}  v^{2/n} \tr( ( L^{(0)})^2) - \frac{2 d_1 d_2}{n} \omega^{-2 d_2/n-1} f(\omega),
\end{align*}
where the polynomial $f(\omega)$ is defined in \eqref{FunctionDeterminingSecondDerivativeOmega}. Therefore, $v^{2/n} \tr( ( L^{(0)})^2)$ can be treated as an independent variable, which is nonnegative, bounded by $\mathscr{F}_0$ and satisfies a well-defined ODE on the $\omega$-limit set $\Omega.$

Since $\mathscr{F}_0$ takes a finite value on $\Omega$ and $\frac{d}{dt} \mathscr{F}_0 = -2 \frac{n-1}{n}  v^{2/n} \tr( ( L^{(0)})^2),$ it follows that $v^{2/n} \tr( ( L^{(0)})^2) \to 0$ as $t \to \infty.$ This in turn implies $f(\omega) \to 0$ and thus $\omega \to \omega_1$ as $t \to \infty$ due to proposition \ref{CharacterisationOfConeSolutionRatio}.

This also implies $\mathscr{F}_0 \to \frac{1}{\omega_1^{2 d_2 / n}} \left( A_1 + A_2 \omega_1^2 - A_3 \omega_1^4 \right)$ as $t \to \infty,$ which is easily seen to be the value of the first cone solution by using the identities in definition \ref{DefinitionConeSolutions}.
\end{proof}

Notice that $\frac{\dot{f}_i}{f_i} = \frac{X_i}{\mathcal{L}} \to \sqrt{\frac{\varepsilon}{2n}}$ as $t \to \infty$ immediately implies that $f_1, f_2$ grow exponentially at infinity. In fact, the metric is asymptotic to the first cone solution at infinity. This follows from a more general result of B\"ohm \cite[Corollary 2.4]{BohmNonCompactEinstein}: If the scalar curvature of the principal orbit is positive and $\mathscr{F}_0$ is bounded from below, then any Einstein trajectory that takes a constant value on $\mathscr{F}_0$ is a cone solution. An argument specifically adapted to the two summands case is given in the proof of proposition \ref{ConvergenceToConeSolution}, see also remark \ref{RemarkConvergenceConeSolutions} (a).

\subsection{Convergence to cone solutions}
\label{SectionConvergenceToConeSolutions}

The results in sections \ref{SubSectionRFMetrics} and \ref{SubSectionAsymptoticsEinsteinScalNeg} show that the non-compact Ricci flat metrics and Einstein metrics with negative scalar curvature of section \ref{CompletenessTwoSummands} are asymptotic to the cone solutions at infinity. In this section it will be shown that the asymptotics of the {\em Ricci flat} trajectories also imply that the metric actually converges to the cone solution as the volume of the singular tends to zero, i.e. as $f_2(0) = \bar{f} \to 0.$ In fact this follows for any sign of the Einstein constant and recovers convergence results due to B\"ohm \cite{BohmInhomEinstein, BohmNonCompactEinstein}. In comparison to B\"ohm's work, the main technical simplification is that the proof does not rely on the Poincar\'e-Bendixson theorem, see also remark \ref{RemarkConvergenceConeSolutions}. 

Recall from \eqref{TwoSummandsMetric} that the metric is given by 
\begin{equation*}
g_{M \setminus Q} = dt^2 +f_1(t)^2 g_S + f_2(t)^2 g_Q
\end{equation*}
away from the singular orbits. It follows from the results of Eschenburg-Wang \cite{EWInitialValueEinstein} that there exists a unique one parameter family $c_{\bar{f}}(t) = (f_1,\dot{f}_1,f_2,\dot{f}_2)(t)$ of solutions to the Einstein equations \eqref{CohomOneRSb}, \eqref{CohomOneRSc} with initial condition $c_{\bar{f}}(0)=(0,1,\bar{f},0)$ for any $\bar{f}>0.$ Moreover, B\"ohm \cite{BohmInhomEinstein} has observed that it depends {\em continuously} on the initial condition $\bar{f}>0.$ Notice also that \eqref{CohomOneRSc} implies $(d_1+1) \ddot{f}_2(0) = \frac{\varepsilon}{2} \bar{f} + \frac{A_2}{d_2} \frac{1}{\bar{f}} >0$ if either $\varepsilon \geq 0$ or $\bar{f}^2< - \frac{2}{\varepsilon} \frac{A_2}{d_2}$ and $\varepsilon < 0.$ However, the equations are a priori not well defined if $\bar{f}=0.$ This singular condition corresponds geometrically to the collapse of the full principal orbit.

To describe the behaviour of the Einstein equations as the volume of the singular orbit tends to zero more concretely, the following observation is key: In the $(X_i,Y_i, \mathcal{L})$-coordinate system defined in \eqref{RescaledTwoSummandsVariables}, the initial condition $(0,1,\bar{f},0)$ of the trajectory $c_{\bar{f}}$ corresponds to the stationary point \eqref{InitialCriticalPoint}, which is independent of $\bar{f}.$ Furthermore, the initial condition $f_2(0)=\bar{f}$ can be recovered via $\bar{f} = \lim_{s \to - \infty} \frac{\mathcal{L}}{Y_2}.$ In particular, $\bar{f} = 0$ is the limit of trajectories with $\mathcal{L} \equiv 0.$ 

However, the two coordinate systems are only equivalent along trajectories with $\mathcal{L} > 0.$ Nonetheless, due to the continuous dependence on the initial condition, any trajectory with $\mathcal{L} \equiv 0$ can hence be viewed as a continuous limit of Einstein trajectories. Hence, the collapse $\bar{f} \to 0$ is described in the $(X_i,Y_i, \mathcal{L})$-coordinates by the solution of the {\em Ricci flat} equations. By construction this solution lies in the unstable manifold  of \eqref{InitialCriticalPoint} and due to proposition \ref{NumberOfParameterFamilies} it is indeed unique.

Furthermore, due to the uniqueness of solutions $c_{\bar{f}}$ of the Einstein equations with initial condition $c_{\bar{f}}(0)=(0,1,\bar{f},0)$, one might expect that the limit as $\bar{f} \to 0$ is a cone solution. This intuition is confirmed in proposition \ref{ConvergenceToConeSolution}. 

\vspace{2mm}

In the case of Einstein metrics with positive scalar curvature, the proof of proposition \ref{ConvergenceToConeSolution} requires the concept of {\em maximal volume orbits:}

Notice that the volume $V$ of the principal orbit satisfies $\dot{V}=V \tr(L),$ where $\tr(L)$ is the mean curvature. Along trajectories corresponding to Einstein metrics with positive scalar curvature, every critical point of $V$ is a maximum or a singular orbit is reached. Therefore, if the maximal volume orbit exists, it is unique and characterised by $\tr(L) = 0.$ 

In the two summands case, if $A_3 = 0,$ due to a result of B\"ohm \cite[section 4, (e)]{BohmInhomEinstein}, the maximal volume orbit always exists. An alternative argument is discussed below, mainly to introduce a natural coordinate system which extends past the maximal volume orbit. 

\begin{lemma}
If $A_3 = 0$ and $\varepsilon <0,$ then any Einstein trajectory has a maximal volume orbit. 
\label{ExistenceMaxVolumeOrbit}
\end{lemma}
\begin{proof} In analogy to \eqref{RescaledTwoSummandsVariables}, introduce the variables
\begin{equation}
\widehat{X}_i =  \frac{\dot{f}_i}{f_i}, \ \widehat{Y}_i  =  \frac{1}{f_i}, \ \text{ for } \ i=1,2, \ \text{ and } \
\widehat{\mathcal{L}} = {\tr(L)}.
\label{HatVariables}
\end{equation}

Due to the assumption $A_3=0$ the two summands Einstein equations take the form
\begin{align*}
\frac{d}{dt} & \widehat{X}_i = - \widehat{X}_i \widehat{\mathcal{L}}+ \frac{A_i}{d_i} \widehat{Y}_i^2 + \frac{\varepsilon}{2} \\
\frac{d}{dt} & \widehat{Y}_i = - \widehat{X}_i \widehat{Y}_i, \\
\frac{d}{dt} & \widehat{\mathcal{L}} = \frac{\varepsilon}{2} - \sum_{i=1}^2 d_i \widehat{X}_i^2
\end{align*}
and the conservation law is
\begin{equation}
\sum_{i=1}^2 d_i \widehat{X}_i^2 + \sum_{i=1}^2 A_i \widehat{Y}_i^2 + (n-1) \frac{\varepsilon}{2} = \widehat{\mathcal{L}}^2.
\label{UnrescaledConsLaw}
\end{equation}
Notice that the time slice has not been rescaled and that the conservation law \eqref{UnrescaledConsLaw} and $\widehat{\mathcal{L}} = \sum_{i=1}^2 d_i \widehat{X}_i$ describe the rescaled Einstein locus \eqref{EinsteinLocus}.

Clearly the above system is an ODE system with polynomial right hand side. In particular, a solution can only develop a finite time singularity if the norm of $( \widehat{X}_i, \widehat{Y}_i, \widehat{\mathcal{L}} )$ blows up. However, the conservation law \eqref{UnrescaledConsLaw} shows that this can only be the case if $\widehat{\mathcal{L}}$ blows up. At the first singular orbit, i.e. at time $t=0,$ one has $\widehat{\mathcal{L}}=+ \infty$ and $\widehat{\mathcal{L}}$ is strictly decreasing for all $t >0$ as $\varepsilon <0.$ Hence, the finite time singularity corresponds to $\widehat{\mathcal{L}}=- \infty$ and in particular there exists a time with $\tr(L) = \widehat{\mathcal{L}} = 0,$ the maximal volume orbit.
\end{proof}

\vspace{2mm}

From now on fix the normalisation $- \frac{\varepsilon}{2} \in \left\{ -n, 0, n \right\}$ of the Einstein constant $- \frac{\varepsilon}{2}$ and recall that in this case the corresponding cone solutions are given by \eqref{ExplicitConeSolutionScalPos}, \eqref{ExplicitConeSolutionScalNonPos}. The following proposition recovers the convergence results of B\"ohm \cite[Theorem 5.7]{BohmInhomEinstein}, \cite[Theorem 11.1]{BohmNonCompactEinstein}.

\begin{proposition} Suppose that $d_1 >1$ and $A_3 = 0.$ As $\bar{f} \to 0,$ the solution $c_{\bar{f}}$ to the two summands Einstein equations converges to the first cone solution on every relatively compact subset of $(0, \pi)$ if $-\frac{\varepsilon}{2}=n$ and $(0, \infty)$ if $-\frac{\varepsilon}{2} \in \left\{ -n, 0 \right\},$ respectively.
\label{ConvergenceToConeSolution}
\end{proposition}
\begin{proof}
Recall that the limit trajectory with $\bar{f} = 0$ corresponds to a trajectory with $\mathcal{L} \equiv 0,$ more precisely the unique solution of the Ricci flat system in $X_i,$ $Y_i$ in the unstable manifold of \eqref{InitialCriticalPoint}. According to proposition \ref{VariablesBoundedRicciFlatTwoSummandsCase}, the Ricci flat trajectory asymptotically approaches the first cone solution, which takes the constant value $X_i = \frac{1}{n}$ and $Y_i=\frac{1}{nc_i}$ for $i=1,2.$ Notice that this is in fact the value at $t=0$ of all cone solutions. Therefore it will be called {\em base point} of the cone solution. 

If $\varepsilon \geq 0$ notice as in the proof of proposition \ref{TwoSummandsEinsteinMetricsWithNegativeScalarCurvature} that the variables $X_i, Y_i$ are bounded, that the B\"ohm functional $\mathscr{F}_0$ is bounded from below and non-increasing, and that it has a critical point on the cone solution. In fact, any Einstein trajectory that takes a constant value on $\mathscr{F}_0$ is a cone solution and $\mathscr{F}_0 = n(n-1) c_{1}^{{2d_1}/{n}} c_{2}^{{2d_2}/{n}}.$ However, since $A_3=0,$ the cone solution is unique and hence the minimum.

If $-\frac{\varepsilon}{2}=n,$ then \eqref{DerivativeOfBohmFunctional} implies that $\mathscr{F}_0$ achieves its minimum along a trajectory $c_{\bar{f}}$ on the maximal volume orbit. On any maximal volume orbit the coordinates \eqref{HatVariables} satisfy $\sum_{i=1}^2 d_i \widehat{X}_i = \widehat{\mathcal{L}}=0$ and the conservation law \eqref{UnrescaledConsLaw} hence implies that the variables $\widehat{X}_i, \widehat{Y}_i$ are bounded. Thus, $\mathscr{F}_0 = n(n-1) \prod_{i=1}^2 \widehat{Y}_i^{-2d_i/n}$ has a minimum on the maximal volume orbit, which is achieved by the value of cone solution. 

However, $\mathscr{F}_0$ is constant on the cone solution and since the solution $c_{\bar{f}}$ approaches the base point of the cone solution as $\bar{f} \to 0$, the claim follows.
\end{proof}

\begin{remarkroman}
(a) The simplifying assumption $A_3=0$ can be relaxed. For the geometric examples in \ref{ExamplesConeSolutionsFromHopfFibrations}, one can calculate directly that the first cone solution realises the minimum. So the exact same proof works if $\widehat{D}>0$ and $\varepsilon \geq 0$ due to proposition \ref{X2VariablePositive}.

(b) The behaviour of the B\"ohm functional close to cone solutions was studied in a more general context in \cite{BohmNonCompactEinstein}. In particular, B\"ohm shows that any {\em stable} cone solution is a local attractor of the cohomogeneity one Einstein equations. In the two summands case, the cone solutions are stable if $d_1>1$. However, the cone solutions corresponding to the circle bundle construction of section \ref{SectionSolitonsFromCircleBundles} are {\em un}stable.

(c) In the original proof, B\"ohm \cite{BohmInhomEinstein} uses a coordinate system specifically adapted to the cone solution to find a limit trajectory, which solves a planar ODE. The limit trajectory lies in a compact planar domain and the Poincar\'e-Bendixson theorem is applied to prove convergence to the base point. Stability of the first cone solution then follows via an attractor function, a version of which is \eqref{LyapunovForNonTrivialBundles} in the Einstein case.

The planar ODE in B\"ohm's work is similar to the reduction of the Ricci flat equations to a planar ODE in section \ref{RicciFlatPlanarSystem}. However, in the Ricci soliton case, the extra degree of freedom of the soliton potential prevents a similar reduction and a different proof is required.

(d) B\"ohm's \cite{BohmNonCompactEinstein} construction of the complete, non-compact Einstein metrics which were recovered in section \ref{CompletenessTwoSummands} relies on the above convergence result, i.e. on the fact that for $f_2(0)= \bar{f} \to 0$ the trajectories remain close to the cone solution and are thus defined for all times. The proof in section \ref{CompletenessTwoSummands} shows moreover that one obtains an Einstein metric for {\em all} $f_2(0)>0.$ Notice that in the Ricci flat case the metric is unique up to scaling.
\label{RemarkConvergenceConeSolutions}
\end{remarkroman}

\subsection{B\"ohm's Einstein metrics of positive scalar curvature}
\label{SectionBohmEinsteinMetricsPosScal}

For the convenience of the reader, this section explains how the refined asymptotics of the {\em Ricci flat} equations in section \ref{RicciFlatPlanarSystem} and proposition \ref{ConvergenceToConeSolution} yield B\"ohm's \cite{BohmInhomEinstein} Einstein metrics of positive scalar curvature on $S^5, \ldots, S^9$ and other low dimensional spaces, including $S^2 \times S^3, \ldots, S^2 \times S^7$ or $S^4 \times S^5$.

It should be emphasised that the overall strategy of the construction due to B\"ohm remains the same. 

\subsubsection{Symmetric solutions}
\label{SymmetricSolutions}

A solution $c_{\bar{f}} =  (f_1,\dot{f}_1,f_2,\dot{f}_2)$ of the two summands Einstein equations with initial condition $c_{\bar{f}}(0)=(0,1,\bar{f},0)$ is called {\em symmetric} if there is $\tau>0$ such that $c_{\bar{f}}(\tau)=(0,-1,\bar{f},0)$. In fact, $c_{\bar{f}}$ is symmetric if and only if there exists $t_0>0$ such that $c_{\bar{f}}(t_0)=(f_1(t_0),0,f_2(t_0),0)$ with $f_1(t_0),f_2(t_0) >0.$ In particular, reflection along the maximal volume orbit, the unique orbit with $\tr(L)=0,$ is an isometry precisely for symmetric solutions.

Moreover, since $\omega = \frac{f_1}{f_2}$ satisfies $\dot{\omega} = \omega( \frac{\dot{f}_1}{f_1} - \frac{\dot{f}_2}{f_2}),$ any symmetric solution is characterised by a critical point of $\omega$ on the maximal volume orbit. It is an important observation due to B\"ohm \cite[Lemma 4.2.1]{BohmInhomEinstein} that critical points of $\omega$ are {\em non-degenerate}. The non-degeneracy of the critical points of $\omega$ allows the application of the following general counting principle. 

\begin{lemma} Let $T_{\bar{f}}, \varepsilon_{\bar{f}}$ be continuous, positive functions of the real parameter $\bar{f}.$ Suppose that $c_{\bar{f}} \colon [0,T_{\bar{f}} + \varepsilon_{\bar{f}}) \to \R^n$ is a family of $C^{1}$-maps which depends continuously on $\bar{f}$ and $\omega \in C^{1}$ is a real valued map such that any critical point of $\omega = \omega \circ c_{\bar{f}}$ is non-degenerate and $\dot{\omega}(0) >0$ for all $\bar{f}.$ 

Let $\mathcal{C}(\bar{f}) = \mathcal{C}(\bar{f}, T_{\bar{f}})$ denote the number of critical points of $\omega$ along $c_f$ before $T_{\bar{f}}.$ Fix $\bar{f}_1 < \bar{f}_2.$ Then the following statements hold:
\begin{enumerate}
\item If $\dot{\omega}(T_{\bar{f}}) \neq 0$ for all $\bar{f} \in [\bar{f}_1, \bar{f}_2]$  then $\mathcal{C}(\bar{f})$ is constant on $[\bar{f}_1, \bar{f}_2].$
\item If $\bar{f}^{*}$ is the unique value of $\bar{f} \in [\bar{f}_1, \bar{f}_2]$ with $\dot{\omega}(T_{\bar{f}})=0$ then
\begin{equation*}
| \mathcal{C}(\bar{f}^{'}) - \mathcal{C}(\bar{f}^{''}) | \leq 1
\end{equation*}
for all $\bar{f}^{'}, \bar{f}^{''} \in  [\bar{f}_1, \bar{f}_2]$ with $\bar{f}^{'} < \bar{f}^{*} < \bar{f}^{''}.$ 
\end{enumerate}
In particular, for any $\bar{f}^{'}, \bar{f}^{''} \in  [\bar{f}_1, \bar{f}_2]$ there exist at least $| \mathcal{C}(\bar{f}^{'}) - \mathcal{C}(\bar{f}^{''}) | $ solutions with $\dot{\omega}(T_{\bar{f}})=0$ for $\bar{f} \in [\bar{f}^{'}, \bar{f}^{''}].$
\label{GeneralCountingArgument}
\end{lemma}

\begin{remarkroman}
(a) In fact, if $c_{\bar{f}}$ is just continuous, the lemma can still be used to count roots of continuous functions along $c_{\bar{f}}.$ In this case $c_{\bar{f}}$ has to intersect the zero set of the function transversally. 

(b) B\"ohm proved the counting principle explicitly in the case where $T_{\bar{f}}$ is the time when the maximal volume orbit is reached, \cite[Lemmas 4.4 and 4.5]{BohmInhomEinstein}. More recently it was used by Foscolo-Haskins in the construction of nearly K\"ahler metrics, \cite[Lemma 7.2]{FHNearlyKaehler}.
\end{remarkroman}

\begin{theorem}[B\"ohm]
Let $d_1 >1,$ $A_3=0$ and $-\frac{\varepsilon}{2} = n.$ If the dimension of the principal orbit satisfies $2 \leq n \leq 8,$ there exist infinitely many symmetric solutions to the two summands Einstein equations.
\label{SymmetricEinsteinMetrics}
\end{theorem}
\begin{proof}
Recall that symmetric solutions are induces by critical points of $\omega$ at the maximal volume orbit. With the normalisation $A_2 = d_2 (d_2-1)>0,$ i.e. in geometric applications $\Ric^Q = d_2-1 > 0,$ the metric of the round sphere $(f_1,f_2)(t) = (\sin(t), \cos(t))$ induces a solution to the two summands Einstein equations without any critical point of $\omega$ before the maximal volume orbit. 

According to lemma \ref{GeneralCountingArgument}, it suffices to show that there are trajectories with an arbitrarily high number of critical points of $\omega$ before the maximal volume orbit. Recall that the maximal volume orbit of a trajectory is achieved exactly when $\tr(L)=0.$ In the $(X_i,Y_i, \mathcal{L})$-coordinates \eqref{RescaledTwoSummandsVariables} this corresponds to the blow up time of $\mathcal{L}.$ In particular, critical points of $\omega$ which are detected by the rescaled system happen to be before the maximal volume orbit. Recall that $\omega{'} = \omega \left( X_1 - X_2 \right)$ and that every critical point in the rescaled variables also corresponds to a critical point of $\omega$ in the original time frame $t.$ Since the Einstein trajectories lie in the subvariety $d_1 X_1 + d_2 X_2 = 1,$ critical points occur if and only if $X_1= \frac{1}{n}.$ 

Recall that by proposition \ref{VariablesBoundedRicciFlatTwoSummandsCase} the trajectory of the {\em Ricci flat} system satisfies $X_i \to \frac{1}{n}$ and $Y_i \to \frac{1}{c_i n}$ where $(c_1,c_2)$ denotes the first cone solution. Moreover, observe that the Ricci flat system is realised by solutions to the two summands system for any value of $\varepsilon \in \R$ by the trajectory with $\mathcal{L} \equiv 0,$ as $\varepsilon$ and $\mathcal{L}$ only occur in the combination $\frac{\varepsilon}{2}\mathcal{L}^2.$ However, as explained in section \ref{SectionConvergenceToConeSolutions}, the limit $\mathcal{L} \equiv 0$ exactly corresponds to a smoothing of the trajectory $c_{\bar{f}}$ in the limit $\bar{f}=0.$ Due to the continuous dependence of the solution on the initial condition, for any $\varepsilon \in \R$ and $\bar{f} >0$ small enough, the solution to the two summands system approaches the base point of the first cone solution along a trajectory which is $C^{0}$-close to the Ricci flat trajectory $\gamma_{\text{RF}}$ of proposition \ref{VariablesBoundedRicciFlatTwoSummandsCase} with $\mathcal{L} \equiv 0,$ and then remains close to the actual cone solution in the sense of proposition \ref{ConvergenceToConeSolution}.

The dimension assumption and corollary \ref{StableSpiral} imply that the projection of the Ricci flat trajectory $\gamma_{\text{RF}}$ onto the $(X_1, Y_1)$-plane rotates infinitely often around the stationary point $(\frac{1}{n}, \frac{1}{c_1 n}),$ which is the base point of the first cone solution. Hence, the variable $X_1$ takes the value $X_1 = \frac{1}{n}$ arbitrarily often, which implies that $\mathcal{C}(\bar{f}, T_{\bar{f}}) \to \infty$ as $\bar{f} \to 0,$ where $T_{\bar{f}}$ denotes the time of the maximal volume orbit.

A direct computation of curvatures shows that the metrics are inhomogeneous and non-isometric, cf. \cite[section 6]{BohmInhomEinstein}.
\end{proof}

As an explicit application, theorem \ref{SymmetricEinsteinMetrics} recovers B\"ohm's Einstein metrics on certain low dimensional spaces \cite[Theorem 3.4]{BohmInhomEinstein}.

\begin{corollary}[B\"ohm]
Let $d_S >1$ and suppose that $Q$ is a compact, connected, isotropy irreducible homogeneous space of positive Ricci curvature and of dimension $d_Q.$ If $2 \le d_S, d_Q$ and $d_S + d_Q \leq 8,$ then there exist infinitely many non-isometric cohomogeneity one Einstein metrics of positive scalar curvature on $S^{d_S+1} \times Q.$
\end{corollary}

\begin{remarkroman}
By considering the linearisation of the Einstein equations along the cone solutions, B\"ohm was also able to construct a symmetric cohomogeneity one Einstein metric on $\mathbb{H}P^{2} \# \overline{\mathbb{H}P}^{2}.$
\end{remarkroman}

\subsubsection{B\"ohm's Einstein metrics on low dimensional spheres}
\label{EinsteinMetricsOnSpheres}

Cohomogeneity one Einstein manifolds with singular orbits of (possibly different) dimensions $d_1,$ $d_2$ can be constructed via solutions $c_{\bar{f}} = (f_1,\dot{f}_1,f_2,\dot{f}_2)$ with $c_{\bar{f}}(0)=(0,1,\bar{f},0)$,  $c_{\bar{f}}(\tau)=(\bar{f}^{'},0,0,-1)$ and $\bar{f}, \bar{f}^{'}, \tau>0$. In the case of $SO(d_1+1) \times SO(d_2+1)$-invariant doubly warped product metrics on spheres, the convergence theory from section \ref{SectionConvergenceToConeSolutions} can be applied to give B\"ohm's \cite{BohmInhomEinstein} inhomogeneous Einstein metrics on $S^5, \ldots, S^9.$ 

For the rest of the section, fix $A_3 = 0$ and normalise the Ricci curvature of the singular orbit to be $\Ric^Q = d_2-1,$ so that $A_i=d_i (d_i-1)$ holds for $i=1,2.$ As before, the trajectory $c_{\bar{f}}$ will always correspond to an Einstein metric on a tubular neighbourhood of a singular orbit of dimension $d_2$ and with a principal orbit of dimension $n=d_1+d_2.$ 

As in section \ref{SectionConvergenceToConeSolutions}, in the $(X_i, Y_i, \mathcal{L})$-coordinates the limit trajectory with $\bar{f} = 0$ corresponds to the unique solution of the Ricci flat system in the unstable manifold of \eqref{InitialCriticalPoint} with $\mathcal{L} \equiv 0$. Recall from proposition \ref{VariablesBoundedRicciFlatTwoSummandsCase} that this trajectory approaches the base point $(\frac{1}{n}, \frac{1}{n c_1})$ of the cone solution asymptotically.  Under the dimensional assumptions $d_1 >1$ and $2 \leq n \leq 8,$ the base point $(\frac{1}{n}, \frac{1}{n c_1})$ is a stable spiral  due to corollary \ref{StableSpiral}. As in the proof of theorem \ref{SymmetricEinsteinMetrics}, it follows from the continuous dependence on the initial value, that also $c_{\bar{f}},$ for $\bar{f}>0$ small enough, exhibits a rotational behaviour as it approaches the base point at $t=0$ of the first cone solution $\gamma.$ Proposition \ref{ConvergenceToConeSolution} then says that given any compact set $K \subset \subset (0, \pi),$ the trajectory $c_{\bar{f}}$ remains $C^{0}$-close to $\gamma$ on $K$ if $\bar{f}>0$ is small enough. 
In fact, $c_{\bar{f}}$ obeys a rotational behaviour in every slice around the cone solution as $\bar{f} \to 0.$ To make this precise, notice that the variables $(\widehat{X}_1,\widehat{Y}_1,\widehat{\mathcal{L}})$ of \eqref{HatVariables} form a local coordinate system along the Einstein trajectories away from the singular orbits. For example, the first cone solution has coordinates $(\cot(t), \frac{1}{c_1 \sin(t)}, n \cot(t)).$ One should think of $\widehat{\mathcal{L}}$ as the time variable. For any fixed value $\widehat{\mathcal{L}}$ and for $\bar{f} >0$ sufficiently small, the trajectory $c_{\bar{f}}$ intersects the $(\widehat{X}_1,\widehat{Y}_1,\widehat{\mathcal{L}})$-plane $P_{\widehat{\mathcal{L}}}$ in a unique point. As $\bar{f}>0$ varies, the intersection points describe a continuous curve in this plane.

\begin{proposition}
Let $A_3 = 0$ and $2 \leq n \leq 8.$ Then in any coordinate slice $P_{\widehat{\mathcal{L}}},$ the intersection points of $c_{\bar{f}}$ with a disc around the first cone solution $\gamma$ in $P_{\widehat{\mathcal{L}}}$ exhibit the same rotational behaviour as $\bar{f} \to 0.$ 
\end{proposition}
This follows from the general counting principle \ref{GeneralCountingArgument} applied to the time $T_{\bar{f}}$ when $c_{\bar{f}}$ intersects the disc, the observation that $\mathcal{C}(c_{\bar{f}},T_{\bar{f}}) \to \infty$ as $\bar{f} \to 0,$ and lemma \ref{OccurrenceOfCriticalPoints} below. 

\begin{lemma} 
Along any trajectory $c_{\bar{f}},$ critical points of $\omega$ occur if and only if $\widehat{X}_1 = \frac{\widehat{\mathcal{L}}}{n}$ and $\omega$ is increasing if $\widehat{X}_1 > \frac{\widehat{\mathcal{L}}}{n}$ and decreasing if $\widehat{X}_1 < \frac{\widehat{\mathcal{L}}}{n}.$ 

Moreover, if $A_3 = 0,$ then the $\widehat{Y}_1$-coordinate of $\omega$ in $P_{\widehat{\mathcal{L}}}$ satisfies $\widehat{Y}_1(\omega) > \widehat{Y}_1(\gamma)$ at any maximum and $\widehat{Y}_1(\omega) < \widehat{Y}_1(\gamma)$ at any minimum, where $\gamma$ is the first cone solution.

\label{OccurrenceOfCriticalPoints}
\end{lemma}
\begin{proof}
The first statement follows from $\dot{\omega} = \omega ( \widehat{X}_1 - \widehat{X}_2 )$ and $\sum_{i=1}^2 d_i \widehat{X}_i = \widehat{\mathcal{L}}.$ If $A_3 = 0,$ the identity $\ddot{\omega} = \frac{A_1}{d_1} \widehat{Y}_1^2 -  \frac{A_2}{d_2} \widehat{Y}_2^2$  holds at every critical point of $\omega.$ However, as $\dot{\omega} = \ddot{\omega} = 0$ only occurs on the cone solution, the claim follows.
\end{proof}

Suppose that $d_{\overline{F}}=(F_1,\dot{F}_1, F_2, \dot{F}_2)$ is also an Einstein trajectory which satisfies $d_{\overline{F}}(0)=(0,1,\overline{F},0)$ but instead induces a metric on a tubular neighbourhood of a singular orbit of dimension $d_2$ and with a principal orbit of dimension $n= d_1+d_2.$ Then the above considerations also apply to $d_{\overline{F}}.$ Clearly, the trajectories depend continuously on the parameters $d_1, d_2>1.$ Hence, in the dimension range $2 \leq n \leq 8,$ the trajectories $c_{\bar{f}}$ and $d_{\overline{F}}$ have the {\em same} rotational behaviour as they approach their respective base point of the cone solution at $t=0.$ 

Now consider the twisted trajectory $d_{\overline{F}}^{\text{twisted}}(t)=(F_2, - \dot{F}_2, F_1, - \dot{F}_1)(t).$ If $\tau >0$ is small enough, then $d_{\overline{F}}^{\text{twisted}}(\tau - t)$ is an actual solution to the Einstein equations due to the symmetries of the equations in $d_1, d_2$ as $A_3=0.$ That is, $d_{\overline{F}}^{\text{twisted}}(t)$ runs through the Einstein equations in `opposite direction', starting at $d_{\overline{F}}^{\text{twisted}}(0) = (\overline{F},0,0,-1)$ and then approaching the same cone solution as $c_{\bar{f}}$ but at the base point corresponding to $t = \pi.$ In particular, it has the {\em opposite} rotational behaviour to $c_{\bar{f}}.$ Due to proposition \ref{ConvergenceToConeSolution}, for $\bar{f}, \overline{F} >0$ small enough, both $c_{\bar{f}}$ and $d_{\overline{F}}^{\text{twisted}}$ intersect the plane $\{(\widehat{X}_1,\widehat{Y}_1, n)\},$ which is the slice of the maximal volume orbit of the cone solution, in a unique point. Since both trajectories in fact wind around the cone solution arbitrarily often as $\bar{f}, \overline{F} \to 0,$ respectively, and $c_{\bar{f}},$ $d_{\overline{F}}^{\text{twisted}}$ have the opposite rotational behaviour, there are infinitely many intersection points in this (or any other) slice. For any such, there exist $t_0, t_1 > 0$ such that the matching condition $c_{\bar{f}}(t_0) = d_{\overline{F}}^{\text{twisted}}(t_1)$ holds. Then 
\begin{align*}
\tilde{c}_{\bar{f},\overline{F}}(t) = \begin{cases} c_{\bar{f}}(t) &  \ \text{ for } \ 0 \leq t \leq t_0 \\
 d_{\overline{F}}^{\text{twisted}}(t_0+t_1-t) & \ \text{ for } \ t_0 \leq t \leq t_0+t_1
\end{cases}
\end{align*}
satisfies $\tilde{c}_{\bar{f},\overline{F}}(0)=(0,1,\bar{f},0)$ and $\tilde{c}_{\bar{f},\overline{F}}(t_0+t_1)=(\overline{F},0,0,-1)$ and it is a {\em smooth} solution to the Einstein equation as required. Smoothness indeed follows from the uniqueness of solutions to ODEs with fixed initial conditions, since $\tilde{c}_{\bar{f},\overline{F}}$ clearly solves the Einstein equations on both intervals. In particular, any such pair $(\bar{f}, \overline{F})$ induces an Einstein metric $g{(\bar{f}, \overline{F})}$ on $S^{n+1}.$ 

\begin{remarkroman}
A direct curvature computation shows that the metrics are indeed inhomogeneous. Moreover, if the metrics $g{(\bar{f}, \overline{F})}$ and $g{(\bar{f}^{'}, \overline{F}{'})}$ on $S^{n+1}$ are isometric, it follows that $\bar{f}=\bar{f}^{'}$ if $d_1 \neq d_2$ and $\bar{f}=\bar{f}^{'}$ or $\bar{f}=\overline{F}{'}$ if $d_1 = d_2,$ since isometries must map orbits onto orbits, see \cite[Section 7]{BohmInhomEinstein}. 
\end{remarkroman}

This recovers B\"ohm's Einstein metrics on low dimensional spheres \cite[Theorem 3.6]{BohmInhomEinstein}:

\begin{corollary}[B\"ohm]
On $S^5$ and $S^6$ there exists one, on $S^7$ and $S^8$ there exist two, and on $S^9$ there exist three infinite families of non-isometric, strictly cohomogeneity one Einstein metrics of positive scalar curvature.
\end{corollary}

\section{Quasi-Einstein Metrics}
\label{SectionQuasiEinsteinMetrics}

\subsection{Introduction}
\label{QEMIntroSection}

In the study of smooth metric measure spaces the $m$-Bakry-\'Emery Ricci tensor $\Ric + \Hess u - \frac{1}{m} du \otimes du$ plays a central role, cf. \cite{CaseSMMSAndQEM}. It also naturally appears in the context of warped product Einstein manifolds, where it has led to the notion of $m$-{\em quasi-Einstein metrics} or $(\lambda, n+m)$-Einstein metrics in the terminology of He-Petersen-Wylie, cf. \cite{HPWUniquenessWarpedProductEinstein}:

\begin{definition}
Let $(M,g)$ be an $n$-dimensional Riemannian manifold, $u \in C^{\infty}(M)$ and $m \in (0,\infty].$ Then $(M,g,e^{-u} d \Vol_M)$ is called {\em $m$-quasi-Einstein manifold} if 
\begin{equation}
\Ric + \Hess u - \frac{1}{m} du \otimes du + \frac{\varepsilon}{2} g = 0.
\label{QEMequation}
\end{equation}
The sum $m+n$ is called {\em effective dimension} and $-\frac{\varepsilon}{2}$ is the {\em quasi-Einstein constant.}
\end{definition}

Kim-Kim \cite{KimKim} observed that any connected $m$-quasi-Einstein manifold with $m < \infty$ satisfies the following conservation law:  There exists a constant $\mu \in \R,$ called {\em characteristic constant,} such that 
\begin{equation}
\Delta u - | \nabla u |^2 + m \mu e^{2u / m} + m \frac{\varepsilon}{2} = 0.
\label{QEMConsLaw}
\end{equation}

In this case, Kim-Kim \cite{KimKim} proved that if $m>1$ is an integer and $(N^m,h)$ is Einstein with $\Ric_h = \mu h$, then the warped product 
\begin{equation}
(M \times N, g + e^{-2u /m} h)
\label{EQWarpedProductKimKim}
\end{equation}
is Einstein. Conversely, if $(M \times N, g + e^{-2u /m} h)$ is Einstein, then $(M,g,e^{-u} d \Vol_M)$ must be $m$-quasi-Einstein.

\vspace{2mm}

This point of view on Einstein warped products was successfully used by Case-Shu-Wei \cite{CSWRigidityQEM} to show that any compact {\em K\"ahler} $m$-quasi-Einstein metric with $m < \infty$ is Einstein. In contrast, recall that all {\em known} non-trivial compact Ricci solitons are K\"ahler. 

Hall \cite{HallQEinstein} constructed $m$-quasi-Einstein metrics on total spaces of complex vector bundles associated to principal circle bundles over products of Fano K\"ahler-Einstein manifolds. Due to the induced hypersurface foliation their geometry can in fact be described using the cohomogeneity one equations from section \ref{CohomOneQEMsection}. The case of a single base factor is due to L\"u-Page-Pope \cite{LuPagePopeQEinstein}. Remarkably, the L\"u-Page-Pope metrics are conformally K\"ahler and the associated K\"ahler class is a multiple of the first Chern class as shown by Batat-Hall-Jizany-Murphy \cite{BHJMConfKahlerQEM}.

\subsection{The initial value problem for cohomogeneity one quasi-Einstein metrics} 
\label{CohomOneQEMsection}

The formulae for the Ricci curvature of a cohomogeneity one manifold in section \ref{SectionCohomOneSetUp} yield that the $m$-quasi-Einstein equation takes the form
\begin{align}
-( \delta^{\nabla^t}L_t)^{\flat} - d(\tr(L_t)) & = 0, \label{QEMequationA} \\
- \tr( \dot{L}_t) - \tr(L_t^2) + \ddot{u} - \frac{1}{m} \dot{u}^2 + \frac{\varepsilon}{2} & =0, \label{QEMequationB} \\
- \dot{L}_t  - (- \dot{u} + \tr(L_t)) L_t  + r_t + \frac{\varepsilon}{2} \mathbb{I} & = 0, \label{QEMequationC} 
\end{align}
and the conservation law \eqref{QEMConsLaw} is given by
\begin{equation}
\ddot{u} + (- \dot{u} + \tr(L)) \dot{u} + m \mu e^{2u/m} + m \frac{\varepsilon}{2} = 0.
\label{CohomOneQEMConsLaw}
\end{equation}

\begin{remarkroman}
Notice that for $f=e^{-u/m}$ the conservation law is equivalent to
\begin{align*}
\frac{d}{dt} \frac{\dot{f}}{f} = - \left( m \frac{\dot{f}}{f} + \tr(L) \right) \frac{\dot{f}}{f} + \frac{\mu}{f^2} + \frac{\varepsilon}{2},
\end{align*}
and hence it is the Einstein equation for the added factor in Kim-Kim's \cite{KimKim} warped product construction \eqref{EQWarpedProductKimKim}.
\label{ConsLawIsEinsteinEQ}
\end{remarkroman}

The following proposition generalises an observation due to Back\cite{BackLocalTheoryofEquiv} in the Einstein case, see also \cite{EWInitialValueEinstein} and  \cite{DWCohomOneSolitons}. 

\begin{proposition}
Let $M$ be a connected manifold and $g$ a $C^2$-Riemannian metric on $M.$ Suppose that $G$ is a compact Lie group which acts isometrically and with cohomogeneity one on $(M,g)$ and that the action has a singular orbit. Let $u \in C^3(M)$ be $G$-invariant. 

Then \eqref{QEMequationC} implies \eqref{QEMequationA} and if the conservation law \eqref{CohomOneQEMConsLaw} is satisfied, then \eqref{QEMequationB} holds as well.
\label{ReducedQEMsystem}
\end{proposition}
\begin{proof}
The fact that \eqref{QEMequationC} implies \eqref{QEMequationA} follows as in the Ricci soliton case, cf.  \cite[Proposition 3.19]{DWCohomOneSolitons}, as the equations are identical. 

Let $v_t$ be the relative volume of the principal orbit $(P,g_t).$ Then it follows that $\frac{d}{dt} v = \tr(L) v$ and due to \cite[Formula (3.16)]{DWCohomOneSolitons} there holds
\begin{equation*}
 \frac{d}{dt} \left( v^2 \left( \Ric(N,N) + \frac{\varepsilon}{2} \right) \right) + v^2 \left( 2 \dot{u} \tr(L^2)+ \frac{d}{dt} \left( \dot{u}\tr(L) \right) \right) = 0.
\end{equation*}
By combining this with $\Ric(N,N)= - \tr(\dot{L}) - \tr(L^2)$ and the conservation law \eqref{CohomOneQEMConsLaw}, one obtains
\begin{equation*}
\frac{d}{dt} \left( v^2 \left( \Ric(N,N) + \ddot{u} - \frac{1}{m} \dot{u}^2 + \frac{\varepsilon}{2} \right) \right)
        = 2 \dot{u} v^2 \left( \Ric(N,N) + \ddot{u} - \frac{1}{m} \dot{u}^2 + \frac{\varepsilon}{2} \right).
\end{equation*}
Therefore $v^2 \left( \Ric(N,N) + \ddot{u} - \frac{1}{m} \dot{u}^2 + \frac{\varepsilon}{2} \right)$ is a multiple of $e^{2 u}$ which vanishes at the singular orbit, and thus vanishes identically.
\end{proof} 

\begin{proposition}
Let $M$ be a smooth manifold of dimension $\dim M \geq 3.$ Suppose that a solution of the $m$-quasi-Einstein equation on $M$ is given by a $C^2$-Riemannian metric $g$ and $u \in C^3(M).$ Then $g$ and $u$ are real analytic in harmonic and geodesic normal coordinates. 
\label{QEMregularity}
\end{proposition}
\begin{proof}
The $m$-quasi-Einstein equation and the contracted second Bianchi identity give rise to the PDE
\begin{align*}
\Ric + \Hess u - \frac{1}{m} du \otimes du + \frac{\varepsilon}{2} g & = 0,  \\
\Delta(du) + \Ric( \cdot, \grad u) - \frac{2}{m} \left( \Delta u \right) du & = 0
\end{align*}
for $(g,u).$ Notice that the $\frac{1}{m}$-terms are of lower order and thus the principal symbol is the same as in the Ricci soliton case. Hence, $(g,u)$ is a solution of a quasi-linear elliptic system and the regularity analysis in\cite[Lemma 3.2]{DWCohomOneSolitons} carries over without any changes, see also \cite[Theorem 5.2]{dTKRegularity}.
\end{proof}

The initial value problem for $m$-quasi-Einstein metrics at a singular orbit can be solved analoguously to Buzano's  \cite{BuzanoInitialValueSolitons} approach in the Ricci soliton case: Due to proposition \ref{ReducedQEMsystem} it suffices to consider \eqref{QEMequationC}, \eqref{CohomOneQEMConsLaw} and the relation $\dot{g}_t = 2 g_t L_t.$ Setting up an ODE system for $(g_t, L_t, u)$ as in \cite{BuzanoInitialValueSolitons}, one observes that the $-\frac{1}{m} \dot{u}^2$-term simply disappears in the error terms that occur in Buzano's proof because it is of lower order. In particular, the construction of a formal power series solution is unchanged. Due to the real analyticity of $m$-quasi-Einstein metrics as in proposition \ref{QEMregularity}, a theorem of Malgrange \cite[Theor{\`e}me 7.1]{MalgrangeEquationsDifferentielle} then yields a genuine solution. Alternatively, a Picard iteration may be applied as in \cite{EWInitialValueEinstein}.

\begin{theorem}
Let $G$ be a compact Lie group acting isometrically on a connected Riemannian manifold $(M,g)$ and suppose there exists a singular orbit $Q = G / H.$ Choose $q \in M$ such that $Q = G \cdot q$ and denote by $V = T_qM / T_qQ$ the normal space of $Q$ at q. Then $H$ acts linearly and orthogonally on $V$ and a tubular neighbourhood of $Q$ may be identified with its normal bundle $E = G \times_H V.$ The principal orbits are $P = G / K = G \cdot v$ for any $v \in V \setminus \left\lbrace 0\right\rbrace.$ These can be identified with the sphere bundle of $E$ (with respect to an $H$-invariant scalar product on $V$). Let $\mathfrak{g} = \mathfrak{h} \oplus \mathfrak{p}_{-}$ be a decomposition of the Lie algebra of $G$ where $\mathfrak{p}_{-}$ is an $Ad_H$-invariant complement of $\mathfrak{h}= \operatorname{Lie}(H).$

Assume that $V$ and $\mathfrak{p}_{-}$ have no common irreducible factors as $K$-representations.

Then for any $\varepsilon \in \R,$ any $m \in (0, \infty],$ any $G$-invariant metric $g_Q$ on $Q$ and any shape operator $L \colon E \to \operatorname{Sym}^2(T^{*}Q)$ there exists a G-invariant $m$-quasi-Einstein metric on some open disc bundle of $E$.
\label{QEMInitialValueTheorem}
\end{theorem}

\begin{remarkroman}
The assumption that $V$ and $\mathfrak{p}_{-}$ have no common irreducible factors as $K$-representations is primarily a technical simplification but as Eschenburg-Wang point out in \cite[Remark 2.7]{EWInitialValueEinstein} it is also natural in the context of the Kaluza-Klein construction.
\end{remarkroman}

\subsection{New quasi-Einstein metrics}
The analysis of the two summands case in section \ref{CompletenessTwoSummands} can be adapted to the $m$-quasi-Einstein case for $m < \infty.$ Recall that the metric restricted to the principal orbit is given by $g_t = f_1(t)^2 g_S + f_2(t)^2 g_Q,$ and set $f_3(t) = e^{-u(t)/m}.$ Due to proposition \ref{ReducedQEMsystem}, it suffices to consider \eqref{QEMequationC} and \eqref{CohomOneQEMConsLaw}, and thus the two summands $m$-quasi-Einstein equations take the form
\begin{align*}
\frac{d}{dt} \left( \frac{\dot{f}_1}{f_1} \right) & = - \tr ( \widehat{L} ) \frac{\dot{f}_1}{f_1} + \frac{\varepsilon}{2} + \frac{A_1}{d_1} \frac{1}{f_1^2} + \frac{A_3}{d_1} \frac{f_1^2}{f_2^4}, \\
\frac{d}{dt} \left( \frac{\dot{f}_2}{f_2} \right) & = - \tr ( \widehat{L} ) \frac{\dot{f}_2}{f_2} + \frac{\varepsilon}{2} + \frac{A_2}{d_2} \frac{1}{f_2^2} - 2 \frac{A_3}{d_2} \frac{f_1^2}{f_2^4}, \\
\frac{d}{dt} \left( \frac{\dot{f}_3}{f_3} \right) & = - \tr ( \widehat{L} ) \frac{\dot{f}_3}{f_3} + \frac{\varepsilon}{2} + \frac{\mu}{f_3^2},
\end{align*}
where $\widehat{L} = \diag \left( \frac{\dot{f}_1}{f_1} \mathbb{I}_{d_1}, \frac{\dot{f}_2}{f_2} \mathbb{I}_{d_2}, \frac{\dot{f}_3}{f_3} \mathbb{I}_{m} \right)$ corresponds to the shape operator in Kim-Kim's \cite{KimKim} warped product construction. Notice that $\widehat{L}$ is only well-defined if $m \in \N,$ but its trace always is.

Due to the regularity theorem \ref{QEMInitialValueTheorem} the metric can be smoothly extended over the singular orbit if the initial conditions 
\begin{align*}
f_1(0)=0, \ \dot{f}_1(0)=1 \ \text{ and } \ f_2(0)= \bar{f} > 0, \ \dot{f}_2(0)=0
\end{align*}
are imposed. Clearly one may fix $u(0)=0$ and then 
\begin{align*}
f_3(0) = 1 \ \text{ and } \ \dot{f}_3(0) = 0 \ \text{ and } \ \ddot{f}_3(0) = \varepsilon + 2 \mu
\end{align*}
are the corresponding smoothness conditions for $f_3$.

Fix $\varepsilon \geq 0$ and $\mu > 0$. It follows that $\dot{f}_i(t) > 0$ for $i=1, 2, 3$ and sufficiently small $t>0.$ In analogy to \eqref{RescaledTwoSummandsVariables}, set 
\begin{align*}
\mathcal{L} = \frac{1}{\tr( \widehat{L} )}, \ \frac{d}{ds} = \mathcal{L} \cdot \frac{d}{dt} \ \text{and} \ X_i = \mathcal{L} \cdot \frac{\dot{f}_i}{f_i}, \ Y_i = \mathcal{L} \cdot \frac{1}{f_i} \ \text{for} \ i=1,2,3.
\end{align*}

In particular, $\mathcal{L}, X_i, Y_i$ are positive initially. Set $d_3 =m.$ Then $\sum_{i=1}^3 d_i X_i =1$ and the rescaled two summands $m$-quasi-Einstein equations take the form
\begin{align*}
X_1^{'} & = X_1 \left(  \sum_{i=1}^3 d_i X_i^2 - \frac{\varepsilon}{2} \mathcal{L}^2 -1 \right)  + \frac{A_1}{d_1}  Y_1^2+\frac{\varepsilon}{2} \mathcal{L}^2 + \frac{A_3}{d_1} \frac{Y_2^4}{Y_1^2}, \\
X_2^{'} & = X_2 \left(  \sum_{i=1}^3 d_i X_i^2 - \frac{\varepsilon}{2} \mathcal{L}^2 -1 \right)  + \frac{A_2}{d_2}  Y_2^2+\frac{\varepsilon}{2} \mathcal{L}^2 -  \frac{2 A_3}{d_2} \frac{Y_2^4}{Y_1^2}, \\
X_3^{'} & = X_3 \left(  \sum_{i=1}^3 d_i X_i^2 - \frac{\varepsilon}{2} \mathcal{L}^2 -1 \right)  + \mu Y_3^2+\frac{\varepsilon}{2} \mathcal{L}^2, 
\end{align*}
\begin{align*}
Y_j^{'} & =Y_j  \left(  \sum_{i=1}^3 d_i X_i^2 - \frac{\varepsilon}{2} \mathcal{L}^2 - X_j \right) \ \text {for} \ j=1,2,3, \\
\mathcal{L}^{'} & =\mathcal{L} \left(  \sum_{i=1}^3 d_i X_i^2 - \frac{\varepsilon}{2} \mathcal{L}^2 \right).
\end{align*}
It follows that $\mathcal{L},$ $X_i$, $Y_i >0$ holds along the flow, except possibly for $X_2.$ In the situations of proposition \ref{X2VariablePositive} and proposition \ref{X2PositiveCircleBundles}, the respective proofs carry over to show that both $X_2 > 0$ and $\frac{Y_2}{Y_1} < \hat{\omega}_1$ are preserved. Since $\hat{\omega}_1^2 < \frac{A_2}{2A_3},$ the conservation law
\begin{align*} 
\sum_{i=1}^3 d_i X_i^2 + A_1 Y_1^2 + A_2 Y_2^2 + m \mu Y_3^2  - A_3 \frac{Y_2^4}{Y_1^2} + (n-1) \frac{\varepsilon}{2} \mathcal{L}^2 = 1
\end{align*}
implies that $X_i$, $Y_i$ are bounded for $i=1,2,3.$ Thus $\mathcal{L}$ cannot blow up in finite time either. Completeness of the metric now follows as in proposition \ref{CompletenessEpsZeroTwoSummands} and corollary \ref{CompletenessEpsPosTwoSummands}. 

If $\varepsilon > 0$ and $\mu = 0,$ then $\mathcal{L},$ $X_i,$ $Y_i$ are bounded due to the conservation law, except possibly $Y_3$. However, the ODE for $Y_3$ implies that $Y_3$ cannot blow up in finite time and a similar argument applies. This shows: 

\begin{theorem}
Let $d_1 \geq 1,$ $A_1=d_1(d_1-1)$ and $(d_1+1)A_2^2 > 4 d_1 d_2 (2d_1+d_2) A_3>0$ and fix $m>0.$

Then the associated two summands ODE gives rise to a $1$-parameter family of complete, non-trivial non-homothetic $m$-Bakry-\'Emery Ricci flat metrics and a $2$-parameter family of non-trivial, complete, non-homothetic $m$-quasi Einstein metrics with quasi-Einstein constant $- \frac{\varepsilon}{2} < 0,$ all of which have positive characteristic constant. 

Furthermore, there exists a $1$-parameter family of complete, non-trivial non-homothetic $m$-quasi-Einstein metrics with quasi-Einstein constant $- \frac{\varepsilon}{2} < 0$ and vanishing characteristic constant.
\label{TwoSummandsQEM}
\end{theorem}

\begin{remarkroman}
Case \cite{CaseNonExistenceQEM} has shown that any complete, non-trivial $m$-Bakry-\'Emery Ricci flat quasi-Einstein manifold has positive characteristic constant.
\end{remarkroman}

Notice that if $A_3 = 0$ and $m \in \N,$ the above construction gives rise to a triply warped product Einstein metric. 

Multiple warped product Einstein metrics of nonpositive scalar curvature were constructed by B\"ohm \cite{BohmNonCompactEinstein} on $\R^{d_1+1} \times M_2 \times \ldots \times M_r$ if $d_1>1,$ for Einstein manifolds $(M_i,g_i)$ of positive scalar curvature $\mu_i >0.$ The corresponding steady and expanding Ricci solitons have been constructed by Dancer-Wang \cite{DWExpandingSolitons, DWSteadySolitons} who in joint work with Buzano and Gallaugher \cite{BDGWExpandingSolitons, BDWSteadySolitons} also settled the case $d_1 =1.$ 

Away from the singular orbit, on $(0, \infty) \times S^{d_1-1} \times M_2 \times \ldots \times M_r,$
the metrics are of the form $dt^2 + \sum_{i=1}^{r} f_i^2(t) g_i.$ Notice that the corresponding $m$-quasi-Einstein equations are 
\begin{align*}
\frac{d}{dt} \frac{\dot{f}_i}{f_i} & = - (- \dot{u} + \tr(L) ) \frac{\dot{f}_i}{f_i} + \frac{\varepsilon}{2} + \frac{\mu_i}{f_i^2} \ \text{ for } \ i=1, \ldots, r,
\end{align*}
where $\mu_1 = d_1-1.$ If $m < \infty,$ then $f_{r+1} = e^{-u/m}$ satisfies an analogous equation, where $\mu_{r+1} > 0$ will be the characteristic constant of the induced $m$-quasi-Einstein metric. 

Due to the regularity theorem \ref{QEMregularity} one induces a smooth $m$-quasi-Einstein metric on the trivial $\R^{d_1+1}$-bundle over $M_2 \times \ldots \times M_r$ by imposing the initial conditions $f_1=0,$ $\dot{f}_1=1$ and $f_i>0$ for $i \geq 2$ at $t=0$ and by requiring that $f_i(t)$ for $i \geq 2$ and $u(t)$ are even. 

Set $d_{r+1} = m$ if $m<\infty$ and $d_{r+1}= \mu_{r+1}=0$ if $m=\infty.$ In terms of the rescaled coordinates $\mathcal{L},$ $X_i,$ $Y_i$ of \eqref{RescaledTwoSummandsVariables} the above initial conditions correspond to the stationary point 
\begin{equation*}
X_1 = \frac{1}{d_1}, Y_1=\frac{1}{d_1} \ \text{and} \ X_i=Y_i=\mathcal{L} = 0 \ \text{for} \ i \geq 2
\end{equation*}
of the Ricci soliton ODE
\begin{align*}
\mathcal{L}^{'} & = \mathcal{L} \left(  \sum_{j=1}^{r+1} d_j X_j^2 - \frac{\varepsilon}{2} \mathcal{L}^2 \right), \\
X_i^{'} &= X_i \left( \sum_{j=1}^{r+1} d_j X_j^2 - \frac{\varepsilon}{2} \mathcal{L}^2 - 1 \right)  + \frac{\varepsilon}{2} \mathcal{L}^2 + \mu_i Y_i^2, \\
Y_i^{'} &= Y_i \left( \sum_{j=1}^{r+1} d_j X_j^2 - \frac{\varepsilon}{2} \mathcal{L}^2 - X_i \right).
\end{align*}

Notice that $f_i(t)>0,$ $\dot{f}_i(t) >0$ for $t >0$ small and thus the rescaled coordinates $\mathcal{L},$ $X_i,$ $Y_i$ are also positive initially. Moreover, for $\varepsilon \geq 0$ positivity is preserved by the flow of the Ricci soliton ODE.

Consider
\begin{align*}
\mathcal{S}_{1, m} = \sum_{i=1}^{r+1} d_i X_i^2 + \sum_{i=1}^{r+1} \mu_i Y_i^2 + (n-1)\frac{\varepsilon}{2} \mathcal{L}^2 -1 \ \text{ and } \
\mathcal{S}_{2,m} = \sum_{i=1}^{r+1} d_i X_i -1.
\end{align*}
In analogy to \eqref{EinsteinLocus}, \eqref{SolitonLocus} it follows that trajectories lying in the preserved locus $\left\{ \mathcal{S}_{1,m} = 0\right\} \cap \left\{ \mathcal{S}_{2,m} = 0\right\}$ correspond to non-trivial $m$-quasi-Einstein metrics for $m < \infty.$  Similarly, non-trivial Ricci soliton metrics correspond to trajectories in $\left\{ \mathcal{S}_{1,\infty} < 0\right\} \cap \left\{ \mathcal{S}_{2,\infty} < 0\right\}$ and Einstein metrics to trajectories in $\left\{ \mathcal{S}_{1,\infty} = 0\right\} \cap \left\{ \mathcal{S}_{2,\infty} = 0\right\}$.

In all cases, if $\varepsilon \geq 0,$ the variables $X_i, Y_i \geq 0$ are bounded, except possibly $Y_1$ if $d_1=1$. However, the ODEs for $\mathcal{L},$ $Y_1$ show as before that $\mathcal{L},$ $Y_1$ cannot blow up in finite time. Completeness of the metric again follows as in proposition \ref{CompletenessEpsZeroTwoSummands} and corollary \ref{CompletenessEpsPosTwoSummands}.

Thus this construction yields $m$-quasi-Einstein metrics on multiple warped products as in Theorem \ref{MainTheoremQEM}, and a unified proof of the works of B\"ohm \cite{BohmNonCompactEinstein} and Buzano-Dancer-Gallaugher-Wang \cite{DWSteadySolitons, DWExpandingSolitons, BDGWExpandingSolitons, BDWSteadySolitons}.


\begin{thebibliography}{Flat}

\bibitem[AK19]{AngenentKnopfRSConicalSingNonuniqueness}
Sigurd~B. Angenent and Dan Knopf, \emph{{Ricci Solitons, Conical Singularities,
  and Nonuniqueness}}, arXiv:1909.08087 (2019).

\bibitem[App18]{AppletonSteadyRS}
Alexander Appleton, \emph{{A family of non-collapsed steady Ricci solitons in
  even dimensions greater or equal to four}}, arXiv:1708.00161 (2018).

\bibitem[Bac86]{BackLocalTheoryofEquiv}
Allen Back, \emph{{Local Theory of Equivariant Einstein Metrics and Ricci
  Realizability on Kervaire Spheres}}, available at \newline
  www.math.cornell.edu/$\sim$back/einstein.ps (1986).

\bibitem[BB82]{BerardBergerySurDeNouvellesEintein}
Lionel B\'erard-Bergery, \emph{Sur de nouvelles vari{\'e}t{\'e}s riemanniennes
  d'{E}instein}, Institut {\'E}lie {C}artan, 6, Inst. {\'E}lie Cartan, vol.~6,
  Univ. Nancy, Nancy, 1982, pp.~1--60.

\bibitem[BDGW15a]{BDGWExpandingSolitons}
Maria Buzano, Andrew~S. Dancer, Michael Gallaugher, and McKenzie Wang,
  \emph{Non-{K}{\"a}hler expanding {R}icci solitons, {E}instein metrics, and
  exotic cone structures}, Pacific J. Math. \textbf{273} (2015), no.~2,
  369--394.

\bibitem[BDGW15b]{BDGWSteadySolitons}
Maria Buzano, Andrew~S. Dancer, Michael Gallaugher, and McKenzie~Y. Wang,
  \emph{{A family of steady Ricci solitons and Ricci-flat metrics}},
  arXiv:1309.6140 (2015), includes numerical studies not published in
  \cite{BDWSteadySolitons}.

\bibitem[BDW15]{BDWSteadySolitons}
M.~Buzano, A.~S. Dancer, and M.~Wang, \emph{A family of steady {R}icci solitons
  and {R}icci-flat metrics}, Comm. Anal. Geom. \textbf{23} (2015), no.~3,
  611--638.

\bibitem[Bes87]{BesseEinstein}
Arthur~L. Besse, \emph{{Einstein manifolds}}, Ergebnisse der Mathematik und
  ihrer Grenzgebiete, 3. Folge $\cdot$ Band 10, Springer-Verlag, Berlin, 1987.

\bibitem[BHJM15]{BHJMConfKahlerQEM}
Wafaa Batat, Stuart~J. Hall, Ali Jizany, and Thomas Murphy, \emph{{Conformally
  K{\"a}hler geometry and quasi-Einstein metrics}}, M{\"u}nster J. Math.
  \textbf{8} (2015), no.~1, 211--228.

\bibitem[B{\"o}h98]{BohmInhomEinstein}
Christoph B{\"o}hm, \emph{{Inhomogeneous Einstein metrics on low-dimensional
  spheres and other low-dimensional spaces}}, Invent. Math. \textbf{134}
  (1998), no.~1, 145--176.

\bibitem[B{\"o}h99]{BohmNonCompactEinstein}
\bysame, \emph{{Non-compact cohomogeneity one Einstein manifolds}}, Bull. Soc.
  Math. France \textbf{127} (1999), no.~1, 135--177.

\bibitem[BS89]{BSExceptionalHolonomy}
Robert~L. Bryant and Simon~M. Salamon, \emph{On the construction of some
  complete metrics with exceptional holonomy}, Duke Math. J. \textbf{58}
  (1989), no.~3, 829--850.

\bibitem[Buz11]{BuzanoInitialValueSolitons}
Maria Buzano, \emph{Initial value problem for cohomogeneity one gradient
  {R}icci solitons}, J. Geom. Phys. \textbf{61} (2011), no.~6, 1033--1044.

\bibitem[Cal79]{CalabiKaehlerMetrics}
E.~Calabi, \emph{M\'{e}triques k\"{a}hl\'{e}riennes et fibr\'{e}s holomorphes},
  Ann. Sci. \'{E}cole Norm. Sup. (4) \textbf{12} (1979), no.~2, 269--294.

\bibitem[Cao96]{CaoSoliton}
Huai-Dong Cao, \emph{Existence of gradient {K}{\"a}hler-{R}icci solitons},
  Elliptic and parabolic methods in geometry ({M}inneapolis, {MN}, 1994), A K
  Peters, Wellesley, MA, 1996, pp.~1--16.

\bibitem[Cas10]{CaseNonExistenceQEM}
Jeffrey~S. Case, \emph{The nonexistence of quasi-{E}instein metrics}, Pacific
  J. Math. \textbf{248} (2010), no.~2, 277--284.

\bibitem[Cas12]{CaseSMMSAndQEM}
\bysame, \emph{Smooth metric measure spaces and quasi-{E}instein metrics},
  Internat. J. Math. \textbf{23} (2012), no.~10, 1250110--1 -- 1250110--36.

\bibitem[Che09]{ChenStrongUniquenessRF}
Bing-Long Chen, \emph{Strong uniqueness of the {R}icci flow}, J. Differential
  Geom. \textbf{82} (2009), no.~2, 363--382.

\bibitem[CSW11]{CSWRigidityQEM}
Jeffrey Case, Yu-Jen Shu, and Guofang Wei, \emph{Rigidity of quasi-{E}instein
  metrics}, Differential Geom. Appl. \textbf{29} (2011), no.~1, 93--100.

\bibitem[CV96]{CVQasuiEinsteinRenormalization}
Thierry Chave and Galliano Valent, \emph{On a class of compact and non-compact
  quasi-{E}instein metrics and their renormalizability properties}, Nuclear
  Phys. B \textbf{478} (1996), no.~3, 758--778.

\bibitem[DHW13]{DHWShrinkingSolitons}
Andrew~S. Dancer, Stuart~J. Hall, and McKenzie~Y. Wang, \emph{Cohomogeneity one
  shrinking {R}icci solitons: an analytic and numerical study}, Asian J. Math.
  \textbf{17} (2013), no.~1, 33--61.

\bibitem[DK81]{dTKRegularity}
Dennis~M. DeTurck and Jerry~L. Kazdan, \emph{Some regularity theorems in
  {R}iemannian geometry}, Ann. Sci. \'Ecole Norm. Sup. (4) \textbf{14} (1981),
  no.~3, 249--260.

\bibitem[DW09a]{DWExpandingSolitons}
Andrew~S. Dancer and McKenzie~Y. Wang, \emph{Non-{K}\"ahler expanding {R}icci
  solitons}, Int. Math. Res. Not. IMRN (2009), no.~6, 1107--1133.

\bibitem[DW09b]{DWSteadySolitons}
\bysame, \emph{Some new examples of non-{K}\"ahler {R}icci solitons}, Math.
  Res. Lett. \textbf{16} (2009), no.~2, 349--363.

\bibitem[DW11]{DWCohomOneSolitons}
\bysame, \emph{On {R}icci solitons of cohomogeneity one}, Ann. Global Anal.
  Geom. \textbf{39} (2011), no.~3, 259--292.

\bibitem[EW00]{EWInitialValueEinstein}
J.-H. Eschenburg and McKenzie~Y. Wang, \emph{The initial value problem for
  cohomogeneity one {E}instein metrics}, J. Geom. Anal. \textbf{10} (2000),
  no.~1, 109--137.

\bibitem[FH17]{FHNearlyKaehler}
Lorenzo Foscolo and Mark Haskins, \emph{New {$G_2$}-holonomy cones and exotic
  nearly {K}\"ahler structures on {$S^6$} and {$S^3\times S^3$}}, Ann. of Math.
  (2) \textbf{185} (2017), no.~1, 59--130.

\bibitem[FIK03]{FIKSolitons}
Mikhail Feldman, Tom Ilmanen, and Dan Knopf, \emph{Rotationally symmetric
  shrinking and expanding gradient {K}\"ahler-{R}icci solitons}, J.
  Differential Geom. \textbf{65} (2003), no.~2, 169--209.

\bibitem[GK04]{GKExpandingRS}
Andreas Gastel and Manfred Kronz, \emph{A family of expanding {R}icci
  solitons}, Variational problems in {R}iemannian geometry, Progr. Nonlinear
  Differential Equations Appl., vol.~59, Birkh{\"a}user, Basel, 2004,
  pp.~81--93.

\bibitem[GPP90]{GPPEinsteinOnSphereR3R4bundles}
G.~W. Gibbons, D.~N. Page, and C.~N. Pope, \emph{{Einstein metrics on
  {$S^3,\;{\bf R}^3$} and {${\bf R}^4$} bundles}}, Comm. Math. Phys.
  \textbf{127} (1990), no.~3, 529--553.

\bibitem[Hal13]{HallQEinstein}
Stuart~James Hall, \emph{Quasi-{E}instein metrics on hypersurface families}, J.
  Geom. Phys. \textbf{64} (2013), 83--90.

\bibitem[Ham88]{HamiltonRFonSurfaces}
Richard~S. Hamilton, \emph{The {R}icci flow on surfaces}, Mathematics and
  general relativity ({S}anta {C}ruz, {CA}, 1986), Contemp. Math., vol.~71,
  Amer. Math. Soc., Providence, RI, 1988, pp.~237--262.

\bibitem[Ham95]{HamiltonSingularites}
\bysame, \emph{The formation of singularities in the {R}icci flow}, Surveys in
  differential geometry, {V}ol.\ {II} ({C}ambridge, {MA}, 1993), Int. Press,
  Cambridge, MA, 1995, pp.~7--136.

\bibitem[HPW15]{HPWUniquenessWarpedProductEinstein}
Chenxu He, Peter Petersen, and William Wylie, \emph{Uniqueness of warped
  product {E}instein metrics and applications}, J. Geom. Anal. \textbf{25}
  (2015), no.~4, 2617--2644.

\bibitem[Ive94]{IveyNewExamplesRS}
Thomas Ivey, \emph{New examples of complete {R}icci solitons}, Proc. Amer.
  Math. Soc. \textbf{122} (1994), no.~1, 241--245.

\bibitem[KK03]{KimKim}
Dong-Soo Kim and Young~Ho Kim, \emph{Compact {E}instein warped product spaces
  with nonpositive scalar curvature}, Proc. Amer. Math. Soc. \textbf{131}
  (2003), no.~8, 2573--2576.

\bibitem[Kob61]{KobayashiCompactFanos}
Shoshichi Kobayashi, \emph{On compact {K}\"ahler manifolds with positive
  definite {R}icci tensor}, Ann. of Math. (2) \textbf{74} (1961), 570--574.

\bibitem[Koi90]{KoisoSoliton}
Norihito Koiso, \emph{On rotationally symmetric {H}amilton's equation for
  {K}\"ahler-{E}instein metrics}, Recent topics in differential and analytic
  geometry, Adv. Stud. Pure Math., vol.~18, Academic Press, Boston, MA, 1990,
  pp.~327--337.

\bibitem[Lau01]{LauretHomogeneousRS}
Jorge Lauret, \emph{Ricci soliton homogeneous nilmanifolds}, Math. Ann.
  \textbf{319} (2001), no.~4, 715--733.

\bibitem[LPP04]{LuPagePopeQEinstein}
H.~L\"u, Don~N. Page, and C.~N. Pope, \emph{New inhomogeneous {E}instein
  metrics on sphere bundles over {E}instein-{K}\"ahler manifolds}, Phys. Lett.
  B \textbf{593} (2004), no.~1-4, 218--226.

\bibitem[Mal74]{MalgrangeEquationsDifferentielle}
Bernard Malgrange, \emph{Sur les points singuliers des \'equations
  diff\'erentielles}, Enseignement Math. (2) \textbf{20} (1974), 147--176.

\bibitem[MS13]{MSgradientRicciSolitons}
Ovidiu Munteanu and Natasa Sesum, \emph{On gradient {R}icci solitons}, J. Geom.
  Anal. \textbf{23} (2013), no.~2, 539--561.

\bibitem[Per02]{Perelman1}
Grisha Perelman, \emph{{The entropy formula for the Ricci flow and its
  geometric applications}}, arXiv:math/0211159 (2002).

\bibitem[PP87]{PagePopeEinsteinMetricsOnCxLineBundles}
Don~N. Page and C.~N. Pope, \emph{Inhomogeneous {E}instein metrics on complex
  line bundles}, Classical Quantum Gravity \textbf{4} (1987), no.~2, 213--225.

\bibitem[PW09]{PWRigidityWithSymmetry}
Peter Petersen and William Wylie, \emph{On gradient {R}icci solitons with
  symmetry}, Proc. Amer. Math. Soc. \textbf{137} (2009), no.~6, 2085--2092.

\bibitem[Sto17]{StolarskiSteadyRSOnCxLineBundles}
Maxwell Stolarski, \emph{{Steady Ricci Solitons on Complex Line Bundles}},
  arXiv:1511.04087 (2017).

\bibitem[WW98]{WWEinsteinS2Bundles}
Jun Wang and McKenzie~Y. Wang, \emph{Einstein metrics on {$S^2$}-bundles},
  Math. Ann. \textbf{310} (1998), no.~3, 497--526.

\end{thebibliography}
\end{document}